\documentclass{siamltex1213}

\usepackage{graphics} 
\usepackage{amsmath} 
\usepackage{amssymb}  

\usepackage{listings}
\usepackage{enumerate}

\usepackage{booktabs}

\usepackage{tikz}

\usepackage{algorithm}
\usepackage{algorithmic}

\usepackage{flushend}

\usepackage{url}

\usepackage[noadjust]{cite}

\usepackage[caption=false]{subfig}

\newtheorem{assumption}[theorem]{Assumption}

\newenvironment{example}{\refstepcounter{theorem} {\em Example} \thetheorem.}{}
\newenvironment{remark}{\refstepcounter{theorem} {\em Remark} \thetheorem.}{}

\newcommand{\derivd}{d}

\newcommand{\mZ}{\mathbb{Z}}
\newcommand{\mC}{\mathbb{C}}
\newcommand{\mR}{\mathbb{R}}
\newcommand{\mT}{\mathbb{T}}

\newcommand{\fC}{\mathfrak{C}}
\newcommand{\fP}{\mathfrak{P}}

\newcommand{\PPC}{\bar{\mathfrak{P}}_+}

\newcommand{\CP}{\mathfrak{C}_+}

\newcommand{\dm}{\derivd m}
\newcommand{\psp}{\derivd \mu}
\newcommand{\LOne}{L^1(\mT^d)}
\newcommand{\LInf}{L^\infty(\mT^d)}

\newcommand{\kb}{\mathbf  k}
\newcommand{\Nb}{\mathbf  N}
\newcommand{\lbb}{\lb}
\newcommand{\lb}{{\boldsymbol\ell}}
\newcommand{\thetab}{{\boldsymbol \theta}}
\newcommand{\E}{{\mathbb E}}
\newcommand{\zetab}{{\boldsymbol \zeta}}

\makeatletter
\def\iddots{\mathinner{\mkern1mu\raise\p@
\vbox{\kern7\p@\hbox{.}}\mkern2mu
\raise4\p@\hbox{.}\mkern2mu\raise7\p@\hbox{.}\mkern1mu}}
\makeatother

\newenvironment{proofWithName}[1]{{\em #1.}}{\endproof}

\title{Multidimensional Rational Covariance Extension with Applications to Spectral Estimation and \\ Image Compression\thanks{This work was supported by the Swedish Research Council (VR), the Swedish Foundation of Strategic Research (SSF), and the Center for Industrial and Applied Mathematics (CIAM).}} 

\author{Axel Ringh\footnotemark[2] \and Johan Karlsson\footnotemark[2] \and Anders Lindquist\footnotemark[3] \footnotemark[2] 
}

\begin{document}

\maketitle
\slugger{sicon}{xxxx}{xx}{x}{x--x}

\renewcommand{\thefootnote}{\fnsymbol{footnote}}
\footnotetext[2]{Division of Optimization and Systems Theory, Department of Mathematics, KTH  Royal Institute of Technology, 100 44 Stockholm, Sweden. (\email{aringh@kth.se}, \email{johan.karlsson@math.kth.se})}
\footnotetext[3]{Departments of Automation and Mathematics, Shanghai Jiao Tong University,
200240 Shanghai, China. (\email{alq@kth.se})}

\renewcommand*{\thefootnote}{\arabic{footnote}}
\setcounter{footnote}{0}

\begin{abstract}
The rational covariance extension problem (RCEP)  is an important problem in systems and control occurring in such diverse fields as control, estimation, system identification, and signal and image processing, leading to many fundamental theoretical questions. In fact, this inverse problem is a key component in many identification and signal processing techniques and plays a fundamental role in prediction, analysis, and modeling of systems and signals. It is well-known that the RCEP can be reformulated as a (truncated) trigonometric moment problem subject to a rationality condition. In this paper we consider the more general multidimensional trigonometric moment problem with a similar rationality constraint. This generalization creates many interesting new mathematical questions and also provides new insights into the original one-dimensional problem. A key concept in this approach is the complete smooth parametrization of all solutions, allowing solutions to be tuned to satisfy additional design specifications without violating the complexity constraints.  As an illustration of the potential of this approach we apply our results to multidimensional spectral estimation and image compression. 
This is just a first step in this direction, and we expect that more elaborate tuning strategies will enhance our procedures in the future.
\end{abstract}



\pagestyle{myheadings}
\thispagestyle{plain}
\markboth{A.~RINGH, J.~KARLSSON, AND A.~LINDQUIST}{MULTIDIMENSIONAL RATIONAL COVARIANCE EXTENSION}

\section{Introduction}\label{sec:intro}

In this paper we consider the (truncated) multidimensional trigonometric moment problem with a certain complexity constraint. 
Many problems in multidimensional systems theory including realization, control, and identification problems, can be cast in this framework \cite{bose2003multidimensional}. Other applications of this type are  image processing \cite{ekstrom1984digital} and spectral estimation in radar, sonar, and medical imaging \cite{stoica1997introduction}.

More precisely, given a set of complex numbers $c_\kb$, $\kb\in\Lambda$,  where $\kb:=(k_1,\ldots, k_d)$ is a vector-valued index belonging to a specified index set $\Lambda\subset \mZ^d$, find a nonnegative bounded measure $\psp$ such that 
\begin{equation} \label{eq:Cov}
c_\kb = \int_{\mT^d} e^{i(\kb,\thetab)} \psp (\thetab) \quad\text{for all $\kb\in\Lambda$},
\end{equation}
where $\mT:=(-\pi,\pi]$, $\thetab:=(\theta_1,\ldots, \theta_d)\in\mT^d$, and $(\kb,\thetab):=\sum_{j=1}^d k_j\theta_j$ is the scalar product in $\mR^d$. 
Moreover, let $e^{i\thetab}:=(e^{i\theta_1},\ldots, e^{i\theta_d})$.
By the Lebesgue decomposition \cite[p. 121]{rudin1987real}, the measure $\psp$ can be  decomposed in a unique fashion as 
\begin{subequations}\label{constrained_measure}
\begin{equation}
\label{dmu}
\psp(\thetab) = \Phi(e^{i\thetab})\dm(\thetab) + \derivd \hat{\mu}(\thetab)
 \end{equation}
into an absolutely continuous part $\Phi\dm$  with spectral density $\Phi$ and Lebesgue measure $$dm (\thetab):=(1/2\pi)^d\prod_{j=1}^d d\theta_j$$
and  a  singular part $\derivd\hat{\mu}$ containing, e.g., spectral lines.
This is an inverse problem, which in general has infinitely many solutions if one exists. A first problem of interest to us in this paper is how to smoothly parametrize the family of all solutions that satisfy the rational complexity constraint 
\begin{equation}\label{eq:rat}
\Phi(e^{i\thetab})=\frac{P(e^{i\thetab})}{Q(e^{i\thetab})},\quad \text{where $P, Q \in {\bar{\fP}_+}\backslash \{0\}$},
\end{equation}
\end{subequations}
where $\fP_+$ is the convex cone of positive trigonometric polynomials 
\begin{equation}
\label{Ptrigpol}
P(e^{i\thetab})=\sum_{\kb\in \Lambda}p_\kb e^{-i(\kb,\thetab)}
\end{equation}
that are positive for all $\thetab\in \mT^d$, and $\bar{\fP}_+$ is its closure; $\fP_+$ will be called the {\em positive cone}.  Moreover, we use the notation  $\partial\fP_+:=\bar{\fP}_+\backslash\fP_+$ for its boundary; i.e., the subset of  $P\in\bar{\fP}_+$ that are zero in at least one point. In this paper we develop a theory based on convex optimization for this problem. 

%

For $d=1$ and $\Lambda=\{0,1,\dots,n\}$ this  trigonometric moment problem with complexity constrains is well understood, and it has a solution with $\derivd \hat{\mu}=0$ if and only if the Toeplitz matrix
\begin{equation*}
T(c) = \begin{bmatrix}
c_0    & c_{-1}  & \ldots & c_{-n} \\
c_1    & c_0     &  & c_{-n+1} \\
\vdots &         & \ddots & \vdots \\
c_n    & c_{n-1} & \ldots & c_0
\end{bmatrix} 
\end{equation*}
is positive definite \cite{LindquistPicci2015}. Such a sequence, $c_0,\ldots, c_n$, will therefore be called a positive sequence in this paper.

In his pioneering work on spectral estimation, J.P. Burg observed that among all spectral densities $\Phi$ satisfying the moment constraints
\begin{subequations}\label{eq:burg}
\begin{equation}
\label{scalarmoments}
c_k =  \int_{\mT} e^{ik\theta} \Phi(e^{i\theta})\frac{d\theta}{2\pi}, \quad k=0,1,\dots,n,
\end{equation}
the one with maximal entropy
\begin{equation}
\label{entropy}
\int_{\mT} \log \Phi(e^{i\theta})\frac{d\theta}{2\pi}
\end{equation}
\end{subequations}
is of the form $\Phi(e^{i\theta})=1/Q(e^{i\theta})$, where $Q(e^{i\theta})$ is a positive trigonometric polynomial \cite{burg1967maximum, burg1975maximum}. Later, in 1981, R.E. Kalman posed the {\em rational covariance extension problem (RCEP)} \cite{kalman1981realization}:
given a finite covariance sequence 
$c_0,\ldots,c_n,$
determine all infinite extensions $c_{n+1},c_{n+2},\ldots$  such that 
\[
\Phi(e^{i\theta})=\sum_{k=-\infty}^{\infty} c_ke^{-ik\theta}
\]
is a positive rational function of degree bounded by $2n$. This problem, which is 
important in systems theory \cite{LindquistPicci2015}, is precisely a (one-dimensional) trigonometric moment problem with the complexity constraint \eqref{eq:rat}. The designation `covariance' emanates from the fact that $c_0,c_1,c_2,\dots,$ can be interpreted as the covariance lags  $\E\{y(t+k) \overline{y(t)} \}=c_k$  of a wide-sense stationary stochastic process $y$ with  spectral density $\Phi$.

In 1983, T.T. Georgiou \cite{georgiou1983partial} (also see \cite{georgiou1987realization}) proved that to each positive covariance sequence and positive numerator polynomial $P$, there exists a rational covariance extension of the sought form \eqref{eq:rat}. He also conjectured that this extension is unique and hence gives a complete parameterization of all rational extensions of  degree bounded by $2n$. This conjecture was first proven in \cite{byrnes1995acomplete}, 
where it was also shown that the complete parameterization is smooth, allowing for tuning. The proofs in \cite{georgiou1983partial,georgiou1987realization,byrnes1995acomplete} were nonconstructive, using topological methods. Later a constructive proof was given in \cite{byrnes1998aconvex,byrnes2001fromfinite}, leading to an approach based on convex optimization. Here $\Phi$ is obtained as the maximizer of a generalized entropy functional
\begin{equation}
\label{eq:entropyScalar}
\int_{\mT} P(e^{i\theta}) \log \Phi(e^{i\theta})\frac{d\theta}{2\pi}
\end{equation}
subject to the moment conditions \eqref{scalarmoments}, and the problem is reduced to solving a dual convex optimization problem. 
Since then, this approach have been extensively studied \cite{georgiou1999theinterpolation, byrnes2001fromfinite, byrnes2001cepstral, byrnes2002identifyability, enqvist2004aconvex, nurdin2006new, lindquist2013thecirculant, ringh2014spectral, ringh2015afast, byrnes2000anewapproach, zorzi2014rational, enqvist2007approximative, pavon2013geometry}, and the approach has also been generalized to a quite complete theory for scalar moment problems \cite{byrnes2001ageneralized, byrnes2003aconvex, georgio2003kullback, byrnes2006generalizedinterpolation, byrnes2006thegeneralized}. Moreover a number of multivariate counterparts, i.e., when $\Phi$ is matrix-valued, have also been solved \cite{ferrante2008hellinger, georgiou2006relative, picci2008modelling, blomqvist2003matrix, ramponi2009aglobally, lindquist2013onthemultivariate, zorzi2014anewfamily, avventi2011spectral}.

A considerable amount of research has also been done in the area of multidimensional spectral estimation; for example, Woods \cite{woods1976two-dimensionalmark}, Ekstrom and Woods \cite{woods1976two-dimensionalspec}, Dickinson \cite{dickinson1980two-dimensional}, and Lev-Ari \emph{et al.} \cite{lev-ari1989multidimensional} to mention a few.
Of special interest is also results by Lang and McClellan \cite{lang1982multidimensional, lang1983spectral, mcclellan1982multi-dimensional, mcclellan1983duality, lang1982theextension, lang1981spectral}, as they consider a similar entropy functional. 
In many of these areas it seems natural to consider rational models.
Nevertheless, the multidimensional version of the RCEP has only been considered at a few instances, for the two-dimensional case in \cite{georgiou2006relative, georgiou2005solution} and in the more general setting of moment problems with arbitrary basis functions in our recent paper \cite{karlsson2015themultidimensional}.

The purpose of this paper is to extend the theory of rational covariance extension from the one-dimensional to the general $d$-dimensional case  and to develop methods for multidimensional spectral estimation. In Section~\ref{sec:mainRes} we summarize the main theoretical results of the paper. This includes the main theorem characterizing the optimal solutions to the weighted entropy functional, which is then proved in Section~\ref{sec:multidimRCEP}. In Section~\ref{sec:corollary} we prove that under certain assumptions the problem is well-posed in the sense of Hadamard and provide comments and examples related to these assumptions. 
In Section~\ref{sec:cepstMatch} we consider simultaneous matching of covariance lags and logarithmic moments, and  Section~\ref{sec:periodic} is devoted to a discrete version of the problem, where the measure $\psp$ consists of discrete point masses placed equidistantly in a discrete grid in $\mT^d$. This is a generalization  to the multidimensional case of recent results in \cite{lindquist2013thecirculant} and is motivated by computational considerations. In fact,  these discrete solutions provide approximations to solutions to moment problems with absolutely continuous measures and allow for fast arithmetics based on the fast Fourier transform (FFT) (cf. \cite{ringh2015afast}). Finally, Sections~\ref{sec:sysIdEx} and \ref{sec:imageComp} are devoted to two examples of how the theory can be applied; the first in system identification and the second in image compression.

\section{Main results}\label{sec:mainRes}

Given the moments $\{c_\kb\}_{\kb\in \Lambda}$, the problem under consideration is to find a positive measure \eqref{constrained_measure} of bounded variation satisfying the moment constraint \eqref{eq:Cov}. 
Let us pause to pin down the structure of the index set $\Lambda$. In view of \eqref{eq:Cov}, we have  $c_{-\kb} = \bar{c}_\kb$, where $\, \bar{} \,$ denotes complex conjugation. Revisiting the one-dimensional result \cite{byrnes2003aconvex,Byrnes-L-09,byrnes2006thegeneralized} for moment problems with arbitrary basis functions, we observe that the theory holds also for sequences with ``gaps'', e.g., for a sequence $c_0, \ldots, c_{k-1}, c_{k+1}, \ldots, c_n$. As seen in \cite{karlsson2015themultidimensional} this observation equally applies to the multidimensional case. Therefore, we shall consider covariance sequences $\{c_\kb\}_{\kb\in \Lambda}$, where $\Lambda \subset \mathbb{Z}^d$ is any finite index set such that $0 \in \Lambda$ and $- \Lambda = \Lambda$. 
We will denote the cardinality of $ \Lambda$ by $|\Lambda|$. Further, let $n_j=\max\{k_j\,|\, \kb\in \Lambda\}$ denote the maximum range of $\Lambda$ in dimension $j$. 

Next, given the inner product 
\begin{equation*}
\langle c, p \rangle = \sum_{\kb \in \Lambda} c_\kb \bar{p}_\kb,
\end{equation*}
we define the open convex cone 
\begin{equation*}
{\fC}_+ := \left\{ c \mid \langle c, p \rangle > 0, \quad \text{for all $P \in \bar{\mathfrak{P}}_+ \setminus \{0\}$} \right\},
\end{equation*}
the closure of which,  $\bar{\fC}_+$, is the dual cone of $\bar{\mathfrak{P}}_+$, with boundary  $\partial \fC_+$. 

We now extend the domain of the generalized entropy functional in \eqref{eq:entropyScalar} to multidimensional nonnegative measures of the type \eqref{constrained_measure} and consider functionals
\begin{equation}\label{primalfunctional}
\mathbb{I}_P(d\mu) = \int_{\mT^d} P(e^{i\thetab}) \log \Phi(e^{i\thetab}) \,\dm(\thetab) ,
\end{equation}
where $\Phi$ is the absolutely continuous part of $d\mu$.%
\footnote{{\color{black}Note that the absolutely continuous part is uniquely defined by the Lebesgue decomposition, and hence the function $\mathbb{I}_P(d\mu)$ is uniquely defined. Moreover, this definition of $\mathbb{I}_P(d\mu)$ can be motivated by the fact that $\lim_{n \to \infty} \int_{\mT^d} \log(\Phi(e^{i\thetab}) + f_n(\thetab)) dm(\thetab) = \int_{\mT^d} \log(\Phi(e^{i\thetab})) dm(\thetab)$ for any log-integrable $\Phi$ and nonnegative ``good kernel" $f_n(\thetab)$ (see, e.g., \cite[p. 48]{stein2003fourier}). See also the discussion in Section \ref{subsec:commentsAndExamples}.}}
This functional is concave, but not strictly concave since the singular part of the measure does not influence the value. This leads to the optimization problem to maximize \eqref{primalfunctional} subject to the moment constraints \eqref{eq:Cov}. Since the constraints are linear,  this is a convex problem. However, as it is an infinite-dimensional optimization problem, it is more convenient to work with the dual problem, which has a finite number of variables but an infinite number of constraints. In fact, the dual problem amounts to minimizing 
\begin{equation}
\label{dualfunctional}
\mathbb{J}_P(Q) = \langle c, q \rangle - \int_{\mathbb{T}^d} P(e^{i\thetab}) \log Q(e^{i\thetab}) \dm
\end{equation}
over all $Q\in\bar{\mathfrak{P}}_+$, and hence $Q(e^{i\thetab})\ge 0$ for all $\thetab\in \mT^d$. Note that \eqref{dualfunctional} takes an infinite value for $Q\equiv 0$. 

\begin{theorem}\label{theo:conjecture}
For every $c \in \mathfrak{C}_+$ and $P \in \bar{\mathfrak{P}}_+ \setminus \{0\}$ the functional \eqref{dualfunctional} is strictly convex and has a unique minimizer $\hat{Q} \in \bar{\mathfrak{P}}_+ \setminus \{0\}$. 
 Moreover, there exists a unique $\hat{c} \in \partial \mathfrak{C}_+$ and a nonnegative singular measure $d \hat{\mu}$
with support $\supp(d \hat{\mu}) \subseteq \{ \thetab \in \mathbb{T}^d  \mid \hat{Q}(e^{i\thetab}) = 0 \}$  such that
\begin{equation*}
c_{\kb} = \int_{\mathbb{T}^d} e^{i (\kb,\thetab)} \left( \frac{P}{\hat Q}  \dm  + d \hat{\mu} \right)\, \text{ for all } \kb \in \Lambda
\end{equation*}
and
\begin{equation*}
\hat{c}_\kb = \int_{\mathbb{T}^d} e^{i (\kb,\thetab)} d \hat{\mu}, \text{ for all } \kb \in \Lambda.
\end{equation*}
For any such $d\hat{\mu}$, the measure $\psp(\thetab) = (P(e^{i\thetab})/\hat{Q}(e^{i\thetab}))\dm(\thetab) + d\hat{\mu}(\thetab)$
 is an optimal solution to the problem to maximize \eqref{primalfunctional} subject to the moment constraints \eqref{eq:Cov}. Moreover, $d \hat{\mu}$ can be chosen with support in at most $|\Lambda|-1$ points.
\end{theorem}

\begin{corollary}
Let $c \in \mathfrak{C}_+$. Then, for any 
\begin{displaymath}
\psp =\frac{P}{Q}\dm,\quad P, Q \in {\bar{\fP}_+}\backslash \{0\}
\end{displaymath}
satisfying the moment condition \eqref{eq:Cov}, $Q$ is the unique minimizer over $\bar{\fP}_+$ of the dual functional \eqref{dualfunctional}.
\end{corollary}

This corollary implies that, for any $c \in \mathfrak{C}_+$, any measure $\psp$ with only absolutely continuous rational part matching $c$ can be obtained by solving \eqref{dualfunctional} for a suitable $P$. However, although $c \in \mathfrak{C}_+$, not all $P$ result in an absolutely continuous solution $\psp=(P/Q)dm$ that satisfies  \eqref{eq:Cov}. Nevertheless, the case when this happens is of particular interest. 

%

\begin{corollary}\label{theo:langMcclellan}
Suppose that  $d\leq 2$. Then, for any $c \in \mathfrak{C}_+$ and $P \in \mathfrak{P}_+$ there exists a $Q \in \mathfrak{P}_+$ such that $\psp=(P/Q)dm$ satisfies \eqref{eq:Cov}. Moreover this $Q$ is the unique solution to the strictly convex optimization problem to minimize the dual functional \eqref{dualfunctional}  over all $Q\in\mathfrak{P}_+$.
\end{corollary}

This result can be deduced from the early work of Lang and McClellan \cite{lang1982multidimensional},
although they do not consider rational solutions explicitly, nor parameterizations of them.
Note  that Corollary~\ref{theo:langMcclellan} is only valid for $P \in \mathfrak{P}_+$, while Theorem \ref{theo:conjecture} holds for all $P \in \bar{\mathfrak{P}}_+ \setminus \{0\}$. This will be further discussed in Section~\ref{sec:corollary}, where the  proof of Corollary~\ref{theo:langMcclellan} will also be concluded.  


\subsection{Covariance and cepstral matching}\label{subsec:mainCeps}
It follows from Theorem~\ref{theo:conjecture} and Corollary~\ref{theo:langMcclellan} 
that $Q$ is completely determined by the pair $(c,P)$. For $d=1$ the choice  $P \equiv 1$ leads to Burg's formulation \eqref{eq:burg}, which has been termed the \emph{maximum-entropy} (ME) solution. 
On the other hand, better dynamical range of the spectrum can be obtained by taking advantage of the extra degrees of freedom in $P$. 
Several methods for selecting $P$ have been suggested in the one-dimensional setting.  Examples are methods  based on inverse problems as in \cite{karlsson2010theinverse, fanizza2008modeling, karlsson2008stability-preserving}, a linear-programming approach as in \cite{byrnes2001cepstral, byrnes2002identifyability}, and simultaneous matching of covariances and cepstral coefficients as in \cite{musicus1985maximum} and independently in \cite{byrnes2001cepstral, byrnes2002identifyability, enqvist2004aconvex, lindquist2013thecirculant}.  Here, in the multivariate setting,  we consider the selection of $P$  based on the simultaneous  matching of logarithmic moments.

We define the (real) {\em cepstrum\/} of a multidimensional spectrum as the (real) logarithm of its absolutely continuous part. The {\em cepstral coefficients\/} are the corresponding Fourier coefficients
\begin{equation}\label{eq:cepstrum}
\gamma_\kb= \int_{\mT^d} e^{i (\kb,\thetab)} \log\Phi(e^{i\thetab}) \dm (\thetab), \mbox{ for } \kb\in \Lambda\setminus \{0\} .
\end{equation}
For spectra that only have an absolutely continuous part this agrees with earlier definitions in the literature (see, e.g., \cite[pp. 500-507]{oppenheim1975digital} or \cite[Chapter 6]{deller2000discrete}).

Given a set of cepstral coefficients 
we now also enforce cepstral matching of the sought family of spectra. This means that we look for $\Phi = P/Q$ that also satisfy \eqref{eq:cepstrum}. Note that the index $\kb=0$ is not included in \eqref{eq:cepstrum}. In fact, for technical reasons, we shall set $\gamma_0=1$. Also to avoid trivial cancelations of constants in $P/Q$, we need to introduce the set
\begin{equation*}
\mathfrak{P}_{+,\circ} := \{ P \in \mathfrak{P}_{+} \; | \; p_{0} = 1 \}.
\end{equation*}

\begin{theorem}\label{theo:cepstral}
Let $\gamma_{\kb}$,  $\kb \in \Lambda\setminus\{0\}$, be any sequence of complex numbers such that $\gamma_{-\kb} = \bar{\gamma}_{\kb}$, and set $\gamma=\{\gamma_{\kb}\}_{\kb \in \Lambda}$ where $\gamma_0=1$. 
Then, for  $c \in \mathfrak{C}_+$,  the convex optimization problem (D) to minimize
\begin{equation}
\label{cevcepstrdual}
\mathbb{J}(P,Q) = \langle c, q \rangle - \langle \gamma, p \rangle + \int_{\mathbb{T}^d} P \log \left(\frac{P}{Q}\right) \dm 
\end{equation}
subject to $(P,Q) \in\bar{\mathfrak{P}}_{+,\circ}\times \bar{\mathfrak{P}}_+$ has an optimal solution $(\hat{P}, \hat{Q})$.
If such a solution  belongs to $\mathfrak{P}_{+, \circ} \times \mathfrak{P}_{+}$, then $\hat\Phi = \hat{P}/\hat{Q}$ satisfies the logarithmic moment condition \eqref{eq:cepstrum} and $d\mu=\hat\Phi dm$ the moment condition \eqref{eq:Cov}.
Moreover, $\hat\Phi$ is also an optimal solution to the  problem (P) to maximize
\begin{equation}
\label{cevcepstrprimal}
\mathbb{I}(\Phi) = \int_{\mathbb{T}^d} \log \Phi \, \dm
\end{equation}
subject to \eqref{eq:Cov} and \eqref{eq:cepstrum} for $d\mu=\Phi dm$.
Finally, if $d\leq 2$, then $\hat{P} \in \mathfrak{P}_{+, \circ}$  implies that $\hat{Q} \in \mathfrak{P}_{+}$.
\end{theorem}

For reasons to become clear in Section~\ref{sec:cepstMatch}, the optimization problems (P) and (D) will be referred to as the primal and dual problem, respectively. 
A drawback with Theorem \ref{theo:cepstral} is that even when $d\leq 2$, a solution to the dual problem can be guaranteed to have a rational spectrum that satisfies \eqref{eq:Cov} and \eqref{eq:cepstrum} only if $\hat{P} \in  \mathfrak{P}_{+, \circ}$. In fact, as we shall see in Section~\ref{sec:cepstMatch}, for a solution with  $\hat{P} \in \partial \mathfrak{P}_{+, \circ}$ we might have $\hat Q\in\partial\mathfrak{P}_+$ and hence covariance mismatch. A remedy in the case $d\leq 2$ is to use the Enqvist regularization,  introduced in the one-dimensional setting in \cite{enqvist2004aconvex}. This makes the optimization problem strictly convex and forces the solution $\hat{P}$ into the set $\mathfrak{P}_{+, \circ}$. In this way we obtain strict covariance matching and approximative cepstral matching. This statement  will be made precise in Theorem \ref{theo:cepstralReg} in Section \ref{subsec:reg}.

\subsection{The circulant covariance extension problem}\label{subsec:periodProb}

In the recent paper \cite{lindquist2013thecirculant}, Lindquist and Picci studied, for the case  $d=1$, the situation when the underlying stochastic process $y(t)$ is periodic. For the $N$-periodic case, the covariance sequence must satisfy the extra condition $c_{N-k}=\bar{c}_k$; i.e.,  the $N\times N$ Toeplitz matrix of one period is Hermitan {\em circulant}.
In this case, the spectral measure must be discrete with point masses at $\zeta_\ell = e^{i\ell\tfrac{2\pi}{N}}$, $\ell=0, 1,\dots,N-1$, on the discrete unit circle, and instead of the moment condition \eqref{eq:Cov}  we have  
\begin{equation}\label{discrete_moments}
c_k = \frac{1}{N} \sum_{\ell = 0}^{N-1}\Phi(\zeta_\ell) \zeta_\ell^{k},
\end{equation}
which is the inverse discrete Fourier transform of the sequence $( \Phi(\zeta_\ell))$.

This was generalized to the multidimensional case in \cite{ringh2015themultidimensional}, where a circulant version of Theorem \ref{theo:conjecture} and Corollary \ref{theo:langMcclellan} was derived.  For $\Nb := (N_1, \ldots, N_d)$, consider the discretization of the $d$-dimensional torus
\begin{displaymath}
\zetab_\lb := (e^{i\ell_1\tfrac{2\pi}{N_1}}, \ldots, e^{i\ell_d\tfrac{2\pi}{N_d}})
\end{displaymath}
where
\begin{displaymath}
\mathbb{Z}^d_{\Nb} := \{ \lb = (\ell_1, \ldots, \ell_d) \; | \; 0 \leq \ell_j \leq N_j-1, j= 1, \ldots, d \},
\end{displaymath}
and define  $\zetab_\lb^\kb=\prod_{j=1}^d\zeta_{\ell_j}^{k_j}$.
Next, let $\fP_+(\Nb)$ be the positive cone of all trigonometric polynomials \eqref{Ptrigpol} 
such that $P(\zetab_\lb) > 0$ for all  $\lb \in \mathbb{Z}^d_\Nb$.  Moreover, define the interior $\mathfrak{C}_+(\Nb)$ of the dual cone as the set of all $\{c_\kb\}_{\kb \in\Lambda}$ such that $\langle c, p \rangle > 0$ for all $P\in\bar{\mathfrak{P}}_+(\Nb) \setminus \{0\}$. Clearly $\fP_+(\Nb)\supset\fP_+$, and hence $\mathfrak{C}_+(\Nb)\subset\mathfrak{C}_+$. Then Theorem 2 and Corollary~3 in \cite{ringh2015themultidimensional} can be combined in the following theorem. 

\begin{theorem}[\cite{ringh2015themultidimensional}]\label{theo:conjectureDisc}
Suppose that $2n_j < N_j$, for $j =1, \ldots, d$, and let  $c \in \mathfrak{C}_+(\Nb)$ and $P \in \bar{\mathfrak{P}}_+(\Nb) \setminus \{0\}$.  Then, there exist a $\hat{Q} \in \bar{\mathfrak{P}}_+(\Nb) \setminus \{0\}$ such that $\hat{Q}$ is a solution to the convex problem to minimize\footnote{Note that limits such as $P\log(Q)$ and $P/Q$ may not be well defined in the multidimensional case, and therefore  we define the expressions $P\log(Q)$ and $P/Q$ to be zero whenever $P=0$. This is not needed in the continuous case as the set where $P$ is zero is of measure zero.}
\begin{equation*}
 \mathbb{J}_P^{\Nb}(Q)=\langle c, q \rangle - \frac{1}{\prod_{j=1}^d N_j}\sum_{\lb \in \mathbb{Z}^d_\Nb} P(\zetab_\lb) \log  Q(\zetab_\lb) 
\end{equation*}
over all $Q \in \bar{\mathfrak{P}}_+(\Nb)$. Moreover, there exists a nonnegative function $\hat{\mu}$ with support $\supp(\hat{\mu}) = \{ \zetab_\lb \, | \, \hat{Q}(\zetab_\lb) = 0,\, \lb\in \mZ^d_\Nb \}$ such that
\begin{equation}\label{discretecov}
c_{\kb} =  \frac{1}{\prod_{j=1}^d N_j} \sum_{\lb \in \mathbb{Z}_\Nb^d} \zetab_\lb^\kb \left( \frac{P(\zetab_\lb)}{\hat{Q}(\zetab_\lb)}  + \hat{\mu}(\zetab_\lb) \right),
\end{equation}
and the number of mass points for $\hat{\mu}$ can be chosen so that at most $|\Lambda|-1$ points $\hat{\mu}(\zetab_\lb)$ are nonzero. Finally, if $P \in \mathfrak{P}_+(\Nb)$ then $\hat{Q}\in\mathfrak{P}_+(\Nb)$, which is then also unique, and hence $\Phi = P/\hat{Q}$ satisfies \eqref{discretecov} with $\hat\mu\equiv 0$.
\end{theorem}

In \cite{lindquist2013thecirculant} it was shown  in the one-dimensional case that as $N \rightarrow \infty$ the solution of the discrete problem, call it $\hat{Q}_N$, converges to the solution to the corresponding continuous problem, call it $\hat{Q}$. This  gives a natural way to compute an approximate solution to the continuous problem using the fast computations of the discrete Fourier transform.  The same holds also in higher dimensions, as seen in the following result.  

\begin{theorem}\label{theo:convergence}
Suppose that $P \in \bar{\mathfrak{P}}_+ \setminus \{0\}$ and $c \in \mathfrak{C}_+$, and let  $\hat{Q}$ and $\hat{Q}_\Nb$ be the optimal solutions of Theorem \ref{theo:conjecture} and Theorem \ref{theo:conjectureDisc}, respectively.  Then  
\begin{equation*}
\lim_{\min(\Nb) \rightarrow \infty} \hat{Q}_\Nb = \hat{Q}
\end{equation*}
uniformly.

\end{theorem}

\section{The Multidimensional rational covariance extension problem}\label{sec:multidimRCEP}

Most of this section will be devoted to proving Theorem \ref{theo:conjecture}. Some technical details are deferred to
the appendix. Possible interpretations of $P$ will be discussed in the end of the section together with an example showing the non-uniqueness of the measure $d\hat{\mu}$.

\subsection{Proof of Theorem \ref{theo:conjecture}}

\subsubsection{Deriving the dual problem}\label{sec.dualderivation}
For a given $P \in \bar{\mathfrak{P}}_+ \setminus \{ 0 \}$ and $c \in \mathfrak{C}_+$, consider the primal problem to maximize \eqref{primalfunctional} subject to the moment constraints \eqref{eq:Cov} over the set of nonnegative bounded measures, i.e., over $\psp=\Phi dm+d\hat\mu$, where  $\Phi$ is a nonnegative $\LOne$ function and $\derivd \hat{\mu}$ is a nonnegative singular measure. The Lagrangian of this problem becomes
\begin{eqnarray*}
\mathcal{L}_P (\Phi, \derivd \hat{\mu}, Q) &= & \int_{\mathbb{T}^d} P \log \Phi  \dm  
+ \sum_{\kb \in \Lambda} \bar{q}_\kb \left(c_\kb - \int_{\mathbb{T}^d} e^{i (\kb, \thetab)} (\Phi \dm + \derivd \hat{\mu}) \right)
\end{eqnarray*}
where $\bar{q}_{\kb}$, $\kb \in \Lambda$, are Lagrange multipliers. Identifying $\sum_{\kb \in  \Lambda} \bar{q}_{\kb} e^{i (\kb, \thetab)}$ with the trigonometric polynomial $Q$, this can be simplified to
\begin{equation*}
\mathcal{L}_P (\Phi, \derivd \hat{\mu}, Q) =\int_{\mathbb{T}^d} \! \! P \log\Phi \, \dm + \langle c, q \rangle - \int_{\mathbb{T}^d} \! \! Q \Phi \dm - \int_{\mathbb{T}^d} \! \! Q \derivd \hat{\mu}.
\end{equation*}
The dual function $\sup_{\psp\ge 0} \mathcal{L}_P (\Phi, \derivd \hat{\mu}, Q)$ is finite only if $Q \in \bar{\mathfrak{P}}_+ \setminus \{0\}$. To see this, let $Q \not \in \bar{\mathfrak{P}}_+$, i.e., suppose there is $\thetab_0 \in \mathbb{T}^d$ for which $Q(\thetab_0) < 0$. 
Then, by letting  $\hat\mu(\thetab_0)\to \infty$ in the singular part $d\hat\mu$, 
we get that $\mathcal{L}_P (\Phi, \derivd \hat{\mu}, Q) \to \infty$. 
Moreover, if $Q \equiv 0$ then since $P$ is continuous and $P \not \equiv 0$ there is a small neighbourhood where $P>0$. Letting $\Phi \to \infty$ in this neighbourhood we again have that $\mathcal{L}_P (\Phi, \derivd \hat{\mu}, Q) \to \infty$. Hence we can restrict the multipliers to $\bar{\mathfrak{P}}_+ \setminus \{0\}$.


Now note that any pair $(\Phi, d\hat\mu)$ maximizing $\mathcal{L}_P (\Phi, \derivd \hat{\mu}, Q)$ must satisfy $\int_{\mathbb{T}^d} Q \derivd \hat{\mu}=0$, or 
equivalently, the support of $d\hat\mu$ is 
contained in $\{\thetab\in \mT^d \,|\, Q(e^{i\thetab})=0\}$. 
Otherwise letting $d\hat\mu=0$ would result in a larger value of the Lagrangian.

Note that the value of the Lagrangian becomes $-\infty$ for any $\Phi$ that vanishes on a set of positive measure, and hence such a $\Phi$ cannot be optimal.
Now, for any direction $\delta \Phi$ such that $\Phi + \epsilon \delta \Phi$ is a nonnegative $\LOne$ function for sufficiently small $\epsilon >0$, consider the directional derivative 
\begin{align*}
 \delta \mathcal{L}_P(\Phi, d\hat\mu, Q ; \delta \Phi) &=\! \lim_{\varepsilon \rightarrow 0} \frac{1}{\varepsilon}\left(\mathcal{L}_P(\Phi\! +\!  \delta \Phi,\! d\hat\mu, \!Q )\! -\! \mathcal{L}_P(\Phi,\! d\hat\mu, \!Q)\right) =
\int_{\mathbb{T}^d} \left( \frac{P}{\Phi} - Q \right) \delta \Phi \dm .
\end{align*}
For a stationary point this must be nonpositive for all feasible directions $\delta \Phi$, and in particular this holds for $\delta \Phi=\Phi\, {\rm sign}(P-Q\Phi)$ which by construction is a feasible direction. For this direction, the constraint becomes $\int_{\mT^d}|P-Q\Phi|dm\le 0$, requiring that $\Phi=P/Q$ a.e.,
which inserted into the dual function yields
\begin{equation}
\label{originaldual}
\sup_{d\mu\ge 0} \mathcal{L}_P(\Phi, d\hat\mu, Q) = \mathbb{J}_P(Q) +  \int_{\mathbb{T}^d} P (\log P - 1) \dm,
\end{equation}
where he last term in  \eqref{originaldual} does not depend on $Q$ and 
\begin{equation}\label{dualfunctional2}
\mathbb{J}_P(Q) = \langle c, q \rangle - \int_{\mathbb{T}^d} P \log  Q \,\dm .
\end{equation}
Hence  the dual problem is equivalent to minimizing $\mathbb{J}_P$ over $\bar{\mathfrak{P}}_+ \setminus \{0\}$.

\subsubsection{Lower semicontinuity of the dual functional}\label{sec:simicont}
For any $Q \in \mathfrak{P}_+$, 
$\mathbb{J}_P(Q)$ is clearly continuous. However, for $Q \in \partial \mathfrak{P}_+$, $\log Q$ will approach $-\infty$ in the points where $Q(e^{i\thetab}) = 0$, and hence we need to consider the behavior of the integral term in \eqref{dualfunctional2}. Since $P$ is a fixed nonnegative trigonometric polynomial, it suffices to consider the integral
$\int_{\mathbb{T}^d} \log Q \, \dm$. 
However, this integral is known as the (logarithmic) Mahler measure of the Laurent polynomial $Q$ \cite{mahler1962onsome}, and it is finite for all  $Q \in \bar{\mathfrak{P}}_+ \setminus \{0\}$ \cite[Lemma 2, p. 223]{schinzel2000polynomials}. This leads to the following lemma, the proof of which 
is deferred to the appendix. 

\begin{lemma}\label{lem:Jcont}
For any $P \in \bar{\mathfrak{P}}_+ \setminus \{0\}$ and $c \in \mathfrak{C}_+$, the functional $\mathbb{J}_P: \bar{\mathfrak{P}}_+ \setminus \{0\} \rightarrow \mathbb{R}$ is lower semicontinuous.
\end{lemma}

\subsubsection{The uniqueness of a solution}
From the first directional derivative 
\begin{equation*}
\delta \mathbb{J}_P(Q; \delta Q) = \langle c, \delta q \rangle - \int_{\mathbb{T}^d} \frac{P}{Q} \delta Q \dm 
\end{equation*}
of the dual functional \eqref{dualfunctional2}, we readily derive the second 
\begin{equation*}
\delta^2 \mathbb{J}_P(Q; \delta Q) = \int_{\mathbb{T}^d} \frac{P}{Q^2} (\delta Q)^2 \dm ,
\end{equation*}
which is clearly nonnegative for all variations $\delta Q$. Therefore, since, in addition, the constraint set $\bar{\mathfrak{P}}_+$ is convex, the dual problem is a convex optimization problem. To see that $\mathbb{J}_P$ is actually strictly convex, note that since $P$ is positive almost everywhere, so is $P/Q^2$. Therefore, for $\delta^2 \mathbb{J}_P(Q; \delta Q)$ to be zero we must have $\delta Q = 0$ almost everywhere, which implies that it is zero everywhere since it is continuous.  This implies that if there exists a solution, this solution is unique.

\subsubsection{The existence of a solution}\label{sec:exist}
If we can show that  $\mathbb{J}_P$ has compact sublevel sets, then $\mathbb{J}_P$ must have a minimum since it is 
lower semicontinuous (Lemma~\ref{lem:Jcont}).

\begin{lemma}\label{lm:compactSublevel}
The sublevel sets $\mathbb{J}_P^{-1}(-\infty, r]$ are compact for all $r \in \mathbb{R}$.
\end{lemma}

For the proof of Lemma \ref{lm:compactSublevel} we need the following lemma modifying  Proposition 2.1 in \cite{byrnes2006thegeneralized} to the present setting.

\begin{lemma}\label{lemma:linearAndLogarithmicGrowth}
For a fixed $c \in \mathfrak{C}_+$, there exists an $\varepsilon > 0$ such that for every $(P,Q) \in (\bar{\mathfrak{P}}_+ \setminus \{ 0\}) \times (\bar{\mathfrak{P}}_+ \setminus \{ 0\})$
\begin{equation}\label{Jestimate}
\mathbb{J}_P(Q) \geq \varepsilon \|Q\|_\infty - \int_{\mathbb{T}^d} P \dm \; \log  \|Q\|_\infty .
\end{equation}
\end{lemma}

\begin{proof}
Since $\langle c,q \rangle$ is a continuous function, it achieves a minimum on the compact set $\{Q \in \bar{\mathfrak{P}}_+ \setminus \{ 0\} \mid \|q\|_{\infty} = 1\}$, where $\|q\|_\infty := \max_{\kb \in \Lambda} |q_{\kb}|$. The minimum value $\kappa_c$ must be positive since $c \in \mathfrak{C}_+$ and hence $\langle c,q \rangle> 0$ for any $q \in \bar{\mathfrak{P}}_+\setminus \{ 0\}$. For any $Q \in \bar{\mathfrak{P}}_+ \setminus \{ 0\}$ we thus have
\begin{equation}\label{eq:mcineq}
\langle c,q \rangle =  \langle c, \frac{q}{\mbox{} \; \, \|q\|_\infty}  \rangle \|q\|_\infty \geq \kappa_c \|q\|_\infty.
\end{equation}
By Lemma~\ref{lem:trigPoly}, $\|Q\|_\infty \leq |\Lambda|\|q\|_\infty$ , and hence by choosing 
$\varepsilon \leq \kappa_c/|\Lambda|$ we get
\begin{equation}\label{eq:cqApprox2}
\langle c,q \rangle \geq \kappa_c \|q\|_\infty \geq \frac{\kappa_c}{|\Lambda|} \|Q\|_\infty \geq \varepsilon \|Q\|_\infty.
\end{equation}
To obtain a bound on the second term in \eqref{dualfunctional2}, we observe that
\begin{displaymath}
\int_{\mathbb{T}^d} \! \! P \log Q \dm =\int_{\mathbb{T}^d} \! \!  P \log \left[ \frac{Q}{\|Q\|_\infty} \right] \dm\!  +\! \int_{\mathbb{T}^d}\! \!  P \dm  \log \|Q\|_\infty 
 \leq \int_{\mathbb{T}^d}\! \!  P \dm \, \log \|Q\|_\infty ,
\end{displaymath}
since $Q/\|Q\|_\infty \leq 1$. Hence \eqref{Jestimate} follows.
 \end{proof}
 
  \begin{proofWithName}{Proof of Lemma \ref{lm:compactSublevel}}
 For any  $r \in \mathbb{R}$, large enough for the sublevel set $\{ Q \in \bar{\mathfrak{P}}_+ \setminus \{ 0\} \mid r \geq \mathbb{J}_P(Q) \}$ to be  nonempty, 
\begin{equation*}
r \geq \mathbb{J}_P(Q) \geq \varepsilon \|Q\|_\infty - \int_{\mathbb{T}^d}\!\! P \dm \, \log  \|Q\|_\infty
\end{equation*}
for some $\varepsilon > 0$ 
(Lemma~\ref{lemma:linearAndLogarithmicGrowth}).
Comparing linear and logarithmic growth we see that the sublevel set is bounded both from above and from below. Moreover, since $\mathbb{J}_P$ is lower semicontinuous (Lemma~\ref{lem:Jcont}), the sublevel sets are also closed \cite[p. 37]{rudin1987real}. Therefore they are compact.
\end{proofWithName}

\subsubsection{Existence of a singular measure}\label{subsubsec:singMeasure}
It remains to show that there exists a measure $d\hat{\mu}$ prescribed by the theorem and that $\psp = P/\hat{Q} \dm + \derivd \hat{\mu}$ is in fact an optimal solution to the primal problem to maximize \eqref{primalfunctional} subject to the moment constraints \eqref{eq:Cov}. To this end, we invoke the KKT-conditions \cite[p. 249]{luenberger1969optimization} for the dual optimization problem, which require that the functional
\begin{equation*}
L_P(Q, d\tilde{\mu}) = \langle c, q \rangle - \int_{\mathbb{T}^d} P \log (Q) \dm - \int_{\mathbb{T}^d} Q d\tilde{\mu}.
\end{equation*}
is stationary at $\hat{Q}$ for some nonnegative measure%
\footnote{Note that by Rietz's representation theorem (for periodic functions), the dual of $C(\mathbb{T}^d)$ is the space of bounded measures on $\mathbb{T}^d$ \cite[p. 133]{luenberger1969optimization}.} 
$d \tilde{\mu}$ and that the  complementary slackness condition $\int_{\mathbb{T}^d} \hat{Q} d\tilde{\mu} = 0$ holds so that $\supp(d \tilde{\mu}) \subseteq \{ \thetab \in \mathbb{T}^d \, | \, \hat{Q}(e^{i\thetab}) = 0 \}$.

Applying the Wirtinger derivatives \cite[pp. 66-69]{remmert1991theory}
\begin{equation}
\label{eq:wirtinger}
\frac{\partial}{\partial z} = \frac{1}{2} \left( \frac{\partial}{\partial x} - i \frac{\partial}{\partial y} \right), 
 \quad \frac{\partial}{\partial \bar{z}} = \frac{1}{2} \left( \frac{\partial}{\partial x} + i \frac{\partial}{\partial y} \right), 
\end{equation}where   $z = x + i\,y$ is a complex variable, we obtain
\begin{equation*}
\frac{\partial L_P(Q, d\tilde{\mu})}{\partial \bar{q}_\kb} = c_\kb - \int_{\mathbb{T}^d} e^{i (\kb,\thetab)} \left( \frac{P}{Q} \dm + d\tilde{\mu} \right),
\end{equation*}
from which we see that a stationary point must satisfy the moment condition \eqref{eq:Cov}. This shows that there exists a singular measure $d \tilde{\mu}$ with the properties prescribed in the first part of the proof, such that $\psp = P/\hat{Q}\dm + d\tilde{\mu}$ matches the covariances, and we may therefore take $d \hat{\mu} = d \tilde{\mu}$. Next, for $\kb\in\Lambda$, we define 
\begin{equation}\label{eq:cHat}
\hat{c}_\kb : = \int_{\mathbb{T}^d} e^{i (\kb,\thetab)} d\hat{\mu} = c_\kb - \int_{\mathbb{T}^d} e^{i (\kb,\thetab)} \frac{P}{\hat{Q}} \dm
\end{equation}
 from which we see that $\hat{c}$ is unique, although $d \hat{\mu}$ might not be.
For a $Q \in \bar{\mathfrak{P}}_+ \setminus \{0\}$,
\begin{equation*}
\begin{array}{l}
\displaystyle \langle \hat{c}, q \rangle = \sum_{\kb \in \Lambda} \hat{c}_\kb \bar{q}_\kb = \sum_{\kb \in \Lambda} \left( \int_{\mathbb{T}^d} e^{i (\kb,\thetab)} d \hat{\mu} \right) \bar{q}_\kb = \int_{\mathbb{T}^d} Q d \hat{\mu} ,
\end{array}
\end{equation*}
which shows that $\langle \hat{c}, q \rangle \geq 0$ for all $Q \in \bar{\mathfrak{P}}_+ \setminus \{0\}$, and thus $\hat{c} \in \bar{\mathfrak{C}}_+$. However, for $\hat{Q}$ we have $\langle \hat{c}, \hat{q} \rangle = \int_{\mathbb{T}^d} \hat{Q} d \hat{\mu}=0$ by  complementary slackness, which shows that $\hat{c} \in \partial \mathfrak{C}_+$. Moreover, it is shown in  \cite{lang1982theextension} that  there exists a discrete representation with support in $|\Lambda|-1$ points for all $\hat{c} \in \partial \mathfrak{C}_+$.
To show that the solution is optimal also for the primal problem we observe that, for all $\psp=\Phi dm+d\hat \mu$,
\[
\mathbb{I}_P(\Phi) \leq \mathcal{L}_P (\Phi, \derivd \hat{\mu}, Q) \leq \mathbb{J}_P(Q) +  \int_{\mathbb{T}^d} P (\log P - 1) \dm .
\]
 Since equality holds for the feasible point $\psp = (P/\hat{Q}) \dm + d \hat{\mu}$, optimality follows. This completes the proof of Theorem~\ref{theo:conjecture}.
 
 An alternative proof of the results in Sections~\ref{sec:simicont}-\ref{sec:exist} can be constructed along the lines of \cite[Section 5]{Ferrante2007further}. In the proof of that paper they use the existence of a coercive spectral density, which in our case follows from the existence of a spectral density in the exponential family \cite{georgiou2006relative}. Also compare this with the proofs of Theorem 5.1 and Theorem 5.2 in \cite{karlsson 2015themultidimensional}, which deals with a more general setting. 
 
 \subsection{Comments and an example}\label{subsec:commentsAndExamples}
In the one-dimensional case it has already been observed that $P$ need not be confined to the cone $\bar{\mathfrak{P}}_+ \setminus \{0\}$ but could be a general nonnegative integrable function with zero locus of measure zero \cite{byrnes2003aconvex,byrnes2006thegeneralized}. 
This fact was implemented in \cite{georgio2003kullback} to interpret the functional \eqref{eq:entropyScalar}  as a Kullback-Leibler pseudo-distance between $P$ and $\Phi$ and hence with $P$ as a Kullback-Leibler prior. In fact, maximizing \eqref{eq:entropyScalar} is equivalent to minimizing the Kullback-Leibler divergence
\[
\mathbb{D}(P \| \Phi ) := \int_{\mT} P \log \left(\frac{P}{\Phi}\right) \dm,
\]
which is nonnegative for functions with the same total mass and equal to zero only when the functions are equal.
In our present more general setting, $P$ could be any absolutely integrable, nonnegative function for which  the set 
$\{\thetab \in \mathbb{T}^d \mid P(e^{i\thetab}) = 0\}$ has measure zero.  In this context it is also possible to interpret the functional \eqref{primalfunctional}  as a Kullback-Leibler distance, not only between the two functions $P$ and $\Phi$, but between the two measures $dp := P \dm$ and $d\mu$. Since $dp$ is absolutely continuous with respect to $d\mu$ we obtain  \cite{renyi1961measures} (see, in particular, equation 3.1)
\[
\int_{\mT^d} P \log \left( \frac{P}{\Phi} \right) \dm = \int_{\mT^d} \log \left( \frac{dp}{d\mu} \right) dp
\]
where $(dp/d\mu) = P/\Phi$ is the Radon-Nikodym derivative.  

Except in the one-dimensional case, the singular part of the measure is in general not unique. To illustrate this fact, we consider the following example in two dimensions, similar to Example~5.4 in \cite{karlsson2015themultidimensional}, where $Q$ has zeros along a line.

\begin{example}
Given $\Lambda = \{(0,0),  \, (-1,0), \, (1,0), \, (0,-1),$ $(0,1), \, (-1,-1), \, (1,1),$ $(-1, 1), \, (1, -1) \}$, consider
\begin{align*}
& P(e^{i\theta_1}, e^{i \theta_2}) = (1 - \cos \theta_1), \\
& \hat{Q}(e^{i\theta_1}, e^{i \theta_2}) = (1 - \cos \theta_1 )(2 - \cos \theta_2).
\end{align*}
Let $c$ be the covariances of the spectrum $\Phi = P/\hat{Q}$, i.e., $c_{0,0} = 1/\sqrt{3}, \, c_{1,0} = 0, \, c_{0,1} = -1 + 2/\sqrt{3}, \, c_{1,1} = 0$ and $c_{-1, 1} = 0$, the remaining covariances being uniquely determined by the conjugate symmetry $c_{-\kb} = \bar{c}_\kb$. Moreover, let  $\hat{c}$ be given by 
\[
\hat{c}_\kb = \int_{\mT^2} e^{i (\kb,\thetab)} \delta(\theta_1) d\theta_1 \frac{d\theta_2}{2\pi}
\]
so that  $\hat{c}_{0,0} = 1, \, \hat{c}_{1,0} = 1, \, \hat{c}_{0,1} = 0, \, \hat{c}_{1,1} = 0$ and $\hat{c}_{-1, 1} = 0$. Clearly $P, \hat{Q} \in \bar{\mathfrak{P}}_+$, and thus $c\in\mathfrak{C}_+$ since  
\[
\langle c, r \rangle = \sum_{\kb \in \Lambda} c_\kb \bar{r}_\kb = \int_{\mT^2} R(e^{i\thetab}) \frac{P(e^{i\thetab})}{\hat{Q}(e^{i\thetab})} dm(\thetab) > 0
\]
for any $R \in \bar{\mathfrak{P}}_+ \setminus \{0\}$.
In the same way, 
\begin{align*}
& \langle \hat{c},\hat{q}\rangle = 
\int_{\mT^2} \hat Q(e^{i\thetab}) \delta(\theta_1) dm(\thetab) 
= \int_{-\pi}^\pi (1-\cos \theta_1)\delta(\theta_1) d\theta_1  \int_{-\pi}^\pi (2-\cos \theta_2)\frac{d\theta_2}{2\pi}  =0,
\end{align*}
and thus $\hat{c}\in\partial\mathfrak{C}_+$. Hence,
$(\hat{Q}, \hat{c})$ is the unique pair prescribed by Theorem \ref{theo:conjecture} for the covariance sequence $c + \hat{c}$ and the numerator polynomial $P$. However, since $\hat{Q}$ is zero for  $\theta_1 = 0$, any measure $d\hat{\mu}$ with support constrained to the line $\theta_1=0$ and mass  $1$ such that $\int_{\mT^2} \cos \theta_2 d\hat{\mu} = 0$ is a solution.
\end{example}

\section{Well-posedness and  counter examples}\label{sec:corollary}
The intuition behind Corollary \ref{theo:langMcclellan} is that the optimal solution $\hat{Q}$ is repelled from the boundary by the following assumption (Assumption \ref{ass:divergingIntegral}) whenever $P \in \mathfrak{P}_+$.
Then, since the measure $\derivd\hat\mu$ can only have mass in the zeros of $Q$, we must have $\derivd\hat\mu = 0$. 

\begin{assumption}\label{ass:divergingIntegral}
The cone $\bar{\fP}_+$ has the property
\begin{equation*}
\int_{\mathbb{T}^d} \frac{1}{Q}\dm(\thetab)  = \infty \quad \text{for all  $Q \in \partial \mathfrak{P}_+$}.
\end{equation*}
\end{assumption}

As noted in \cite{byrnes2006thegeneralized}, Assumption~\ref{ass:divergingIntegral} always holds in the one-dimensional case ($d=1$), since the trigonometric functions are Lipschitz continuous. Using  results by Georgiou  \cite[p. 819]{georgiou2005solution} it can be shown that this assumption is also valid for  $d = 2$. However,  Lang and McClellan \cite{lang1982multidimensional} note that Assumption~\ref{ass:divergingIntegral} does not hold in general for dimensions $d \geq 3$. To see this, they consider the polynomial $Q(e^{i\thetab}) = \sum_{\ell = 1}^{d} (1 - \cos\theta_\ell )\in \partial\mathfrak{P}_{+} $ and show that $\int_{\mT^d} \frac{1}{Q} dx < \infty$ for $d \geq 3$. In fact, we have the following amplification of this fact, the proof of which we defer to the appendix.

\begin{proposition}\label{prop:d>2}
For $d\geq 3$, Assumption~\ref{ass:divergingIntegral} does not hold if the index set $\Lambda$ contains at least three linearly independent vector-valued indices.
\end{proposition}

Observe that a problem of dimension $d\geq 3$ for which $\Lambda$ contains less than three linearly independent vector-valued indices trivially reduces to a problem in one or two dimensions. Hence in general we identify Assumption~\ref{ass:divergingIntegral} with the case $d\leq 2$.
Corollary \ref{theo:langMcclellan} now follows directly from the following lemma.

\begin{lemma}\label{lem:noSolOnBoundary}
Let $P \in \mathfrak{P}_+$, and suppose that Assumption \ref{ass:divergingIntegral} holds. Then the optimal solution $\hat Q$ to the problem to minimize \eqref{dualfunctional} over all $Q\in\bar{\mathfrak{P}}_+$ belongs to $ \mathfrak{P}_+$.
\end{lemma}

\begin{proof} 
Let $Q \in \partial \mathfrak{P}_+$ be arbitrary. Then, for any $\rho > 0$,  $Q(e^{i\thetab}) + \rho > 0$ for all $\thetab\in\mT^d$.
 Hence the functional $\mathbb{J}_P$ is also differentiable in $Q + \rho$, and the directional derivative in the direction $1$ is 
\begin{equation*}
\delta \mathbb{J}_P(Q + \rho; 1) = \langle c, 1\rangle - \int_{\mathbb{T}^d} \frac{P}{Q + \rho} \dm.
\end{equation*}
Now note that $P/(Q + \rho)$ is nonnegative in all points, that it is pointwise monotone increasing for decreasing values of $\rho$, and that it converges pointwise in extended real-valued sense%
\footnote{
That is, the limit may be $\infty$.
}  
to $P/Q$. Hence by Lebesgue's monotone convergence theorem \cite[p. 21]{rudin1987real} we have, as $\rho\to 0$, 
\begin{equation*}
\int_{\mathbb{T}^d} \frac{P}{Q + \rho} \dm \; \longrightarrow \; \int_{\mathbb{T}^d} \frac{P}{Q} \dm  ,
\end{equation*}
which, since $P \in \mathfrak{P}_+$, is infinite by Assumption \ref{ass:divergingIntegral}. Therefore $1$ is a  descent direction from the point $Q$, and hence the optimal solution is not obtained there. Since $Q \in \partial \mathfrak{P}_+$ is arbitrary, this means that the optimal solution is not attained on the boundary, i.e., we have $\hat Q \in\mathfrak{P}_+$.
\end{proof}

 It turns out that the multidimensional rational covariance extension problem for $d\leq 2$ is in fact well-posed in the sense of Hadamard, i.e., the solution depends  smoothly on $c$ and $P$, which is an important property when it comes to tuning of solutions to design specifications. This follows from the following generalizations to the multidimensional case of Theorems 1.3 and 1.4 in \cite{byrnes2006thegeneralized}, proved in 
the appendix.

\begin{theorem}\label{theo:fp}
Let $f^p: \mathfrak{P}_+ \rightarrow \mathfrak{C}_+$ be the map from $Q$ to $c$, given component-wise by
\begin{equation*}
c_\kb = \int_{\mathbb{T}^d} e^{i (\kb, \thetab)} \frac{P}{Q} \dm
\end{equation*}
for a fixed $P \in \mathfrak{P}_+$. If $d\leq 2$, $f^p$ is a diffeomorphism.
\end{theorem}

\begin{theorem}\label{theo:gc}
Suppose that $d\leq 2$. Let  $f^p$ be as in Theorem \ref{theo:fp}, and let $c \in \mathfrak{C}_+$ be fixed. Then the function $g^c: \mathfrak{P}_+ \rightarrow \mathfrak{P}_+$ mapping $P$ to $Q = (f^p)^{-1}(c)$  is a diffeomorphism onto its image $\mathfrak{Q}_+$.
\end{theorem}


By Corollary \ref{theo:langMcclellan}, the unique solution $\hat Q$ of the dual problem belongs to the interior $\mathfrak{P}_+$ for every pair $(c, P) \in \mathfrak{C}_+ \times \mathfrak{P}_+$ if Assumption \ref{ass:divergingIntegral} holds.
Note  that, while the more general Theorem \ref{theo:conjecture} holds for all $P \in \bar{\mathfrak{P}}_+ \setminus \{0\}$, Corollary~\ref{theo:langMcclellan} is only valid for $P \in \mathfrak{P}_+$.
The reason for this is that if $P \in \mathfrak{P}_+$ the directional derivative
of $\mathbb{J}_P$ tends to $-\infty$ on the boundary by Assumption \ref{ass:divergingIntegral}, so a minimum is not attained there, as we just saw in the proof of Lemma~\ref{lem:noSolOnBoundary}.  
On the other hand, if $P \in \partial \mathfrak{P}_+$, we have $\int_{\mathbb{T}^d} (P/Q) \dm < \infty $ for some $Q \in \partial \mathfrak{P}_+$, take for example $Q = P$. More generally, the integral may not diverge if the zeros of $Q$ belong to a subset of the zeros of $P$. In this case, there is no guarantee that the optimal solution is an interior point. 
The following simple one-dimension example  illustrates this.

\begin{example}\label{ex:1Dexamplesingular}
Consider a one-dimensional problem of degree one, i.e., with $\Lambda=\{-1,0,1\}$. Fix $c=(1,c_1)$, where $c_1\in (-1,0)$ is arbitrary.  Clearly the Toeplitz matrix {\color{black} $T(c)=\left[c_{k-\ell}\right]_{k,\ell=0}^n$} is positive definite, and hence $c \in \mathfrak{C}_+$. We fix $P(e^{i\theta})=2+e^{i\theta}+e^{-i\theta}$, 
which belongs to $\partial \mathfrak{P}_+$  since $P(e^{i\pi}) = 0$. We want to find a $Q\in\mathfrak{P}_+$ of degree at most one so that $\Phi=P/Q$  matches the covariance sequence $c$, i.e, 
\begin{equation}
\label{matching2c}
c_k=\int_\mT e^{ik\theta}\frac{P}{Q}dm, \quad k=0,1.
\end{equation}
Any such $Q$ must have the form $Q(e^{i\theta})=\lambda(1-\rho e^{i\theta})(1-\bar{\rho} e^{-i\theta})$
for some $\lambda>0$ and $|\rho|<1$. Now, clearly
\begin{equation*}
\Phi(e^{i\theta})= \lambda^{-1}\frac{2+e^{i\theta}+e^{-i\theta}}{1 - |\rho|^2} \frac{1 - |\rho|^2}{(1-\rho e^{i\theta})(1-\bar{\rho} e^{-i\theta})},
\end{equation*}
where the second factor takes the form
\begin{equation*}
\frac{1}{1-\rho e^{i\theta}}+\frac{1}{1-\bar\rho e^{-i\theta}}-1 \vspace*{2pt} 
= \cdots+\bar\rho^2 e^{-2i\theta}+\bar \rho e^{-i\theta}+1+\rho e^{i\theta}+\rho^2 e^{2i\theta}+\cdots,
\end{equation*}
which implies that $c_0= \lambda^{-1}(2+\rho+\bar \rho)(1-|\rho|^2)^{-1}$ and $c_1= \lambda^{-1}(1+\rho)^2(1-|\rho|^2)^{-1}$. 
Since $c_0 = 1$, we have $c_1=(1+\rho)^2(2+\rho+\bar \rho)^{-1}$,
which has positive, real denominator. Then, since $c_1<0$, $1+\rho$ is purely imaginary, which is impossible since $1+\rho$ has a positive real part. Hence, there is no $Q\in\mathfrak{P}_+$ of degree at most one  satisfying \eqref{matching2c}. However, for a certain $Q\in\partial \mathfrak{P}_+$, namely $Q(e^{i\theta}) = (2+e^{i\theta}+e^{-i\theta})/(1 + c_1)$,
we obtain $\psp = (P/Q)\dm - c_1 \delta(\theta-\pi)d\theta$, i.e., 
\[
\psp = (1 + c_1) \dm - c_1 \delta(\theta-\pi)d\theta,
\]
which matches $c$ with $-1 < c_1 < 0$.  Now  $\Phi=1 + c_1$ and the singular measure $d\hat\mu=\delta(\theta-\pi)d\theta$ has all its mass at the zero of $Q$, as required by Theorem~\ref{theo:conjecture}.

In this context it is interesting to note that  the covariance extension problem is usually formulated as a partial realization problem where one wants to determine an extension of the partial covariance sequence $c$ so that 
\begin{displaymath}
\Phi_+(z)=\frac12c_0+\sum_{k=1}^\infty c_kz^{-k}
\end{displaymath}
is positive real, i.e., $\Phi_+$ maps the unit disc to the right half of the complex plane; see, e.g., \cite{LindquistPicci2015}. Then $\Phi_+(e^{i\theta})+\Phi_+(e^{i\theta})^*$ is the corresponding spectral density $\Phi(e^{i\theta})$. In our example such a solution is provided by
\begin{displaymath}
\Phi_+(z)=\frac12\left(1+c_1-c_1\frac{1-z}{1+z}\right)=\frac12+c_1z-c_1z^2+\cdots,
\end{displaymath}
yielding precisely $\Phi=1 + c_1$. The singular measure never appears in this framework. 
\end{example}

\section{Logarithmic moments and cepstral matching}\label{sec:cepstMatch}

Given $c\in\mathfrak{C}_+$,  Corollary \ref{theo:langMcclellan} and Theorem~\ref{theo:gc}  together provide a complete smooth parameterization in terms of $P\in\mathfrak{P}_+$ of all $\Phi=P/Q$ such that $d\mu=\Phi dm$ satisfies the moment equations \eqref{eq:Cov}. Therefore the solution can be tuned to satisfy additional design specification by adjusting  $P$. How to determine the best $P$ is, however, a separate problem. 
Theorem \ref{theo:cepstral}, to be proved next, extends results from the one-dimensional case to simultaneously  estimate $P$ using the cepstral coefficients and logarithmic moment matching.

\begin{proofWithName}{Proof of Theorem \ref{theo:cepstral}}
The proof follows along the same lines as that of Theorem 
\ref{theo:conjecture}. By relaxing the primal problem (P)  we get the Lagrangian
\begin{eqnarray}
\mathcal{L} (\Phi, P, Q) &=& \int_{\mathbb{T}^d} \log\Phi \, \dm
 + \sum_{\kb \in \Lambda} \bar{q}_\kb \left(c_\kb - \int_{\mathbb{T}^d} e^{i (\kb,\thetab)} \Phi\, \dm \right)\label{eq:LagrangianCeps}\\
 &+&\sum_{\kb \in \Lambda\setminus\{ 0\}}\bar{p}_\kb \left(\int_{\mathbb{T}^d} e^{i (\kb,\thetab)} \log\Phi \, \dm - \gamma_\kb \right),\nonumber
\end{eqnarray}
where $\bar{q}_{k}$ and $\bar{p}_{k}$ are Lagrangian multipliers. Setting $ p_0 =\gamma_0= 1$ and rearranging terms, this can be written as
\begin{equation}
\label{LagrangianCeps2}
\mathcal{L}(\Phi, P, Q) = \langle c, q \rangle - \int_{\mathbb{T}^d} Q \Phi \,\dm
- \langle \gamma, p \rangle + 1 + \int_{\mathbb{T}^d} P \log \Phi \,\dm,
\end{equation}
where the first term in \eqref{eq:LagrangianCeps} has been incorporated in the last term of \eqref{LagrangianCeps2}.
As before, $\sup_{\Phi\ge 0} \mathcal{L} (\Phi, P, Q)$ is only finite if we restrict $Q$ to  $\bar{\mathfrak{P}}_+$, and similarly we need to restrict $P$ to  $\bar{\mathfrak{P}}_{+,\circ}$. Taking the directional derivative of \eqref{LagrangianCeps2} in any direction $\delta \Phi$ such that $\Phi + \epsilon \delta \Phi$ is a nonnegative $\LOne$ function for all $\varepsilon\in (0,a)$ for a sufficiently small  $ a>0$, we obtain
\begin{equation*}
\delta \mathcal{L}(\Phi, P, Q ; \delta \Phi) = \int_{\mathbb{T}^d} ( P \frac{1}{\Phi} - Q ) \delta \Phi \dm.
\end{equation*}
For the directional derivative to be nonpositive for all feasible directions $\delta \Phi$ we need $\Phi=P/Q$ a.e. (cf. Section~\ref{sec.dualderivation}),
which inserted into \eqref{LagrangianCeps2} yields 
\begin{equation}\label{eq:tmp1}
\sup_{\Phi} \mathcal{L} (\Phi, P, Q) = \mathbb{J}(P,Q) + 1 - \int_{\mathbb{T}^d} P \dm ,
\end{equation}
with $\mathbb{J}(P,Q)$ given by \eqref{cevcepstrdual}. 
A closer look at the last term in \eqref{eq:tmp1} shows that
\begin{displaymath}
\int_{\mathbb{T}^d} \hspace*{-2pt} P \dm = \int_{\mathbb{T}^d} \sum_{\kb \in \Lambda} p_\kb  e^{i (\kb, \thetab)} \dm 
=\sum_{\kb \in \Lambda} p_\kb  \prod_{j=1}^d\int_{-\pi}^{\pi} e^{i k_j \theta_j} \frac{\derivd \theta_j}{2\pi} =1,
\end{displaymath}
since all integrals vanish  except those for $k_1 = \ldots = k_d = 0$. Consequently, $\mathbb{J}$ is precisely the dual functional \eqref{eq:tmp1}. 

Using the Wirtinger derivative from \eqref{eq:wirtinger} to form the gradient of $\mathbb{J}$, we obtain 
\begin{subequations}
\begin{equation}
\label{dJ/dq}
\frac{\partial \mathbb{J}(P,Q)}{\partial \bar{q}_{\kb}} = c_\kb - \int_{\mathbb{T}^d} \!\!   e^{i (\kb,\thetab)} \frac{P}{Q}\, \dm,  \quad \kb\in\Lambda ,\phantom{xxxxxxxxxx}
\end{equation}
\begin{equation}
\label{dJ/dp}
\frac{\partial \mathbb{J}(P,Q)}{\partial \bar{p}_\kb} =  \int_{\mathbb{T}^d} \!\!  e^{i (\kb, \thetab)} \log \left( \frac{P}{Q} \right) \dm - \gamma_\kb, \quad \kb\in\Lambda\setminus\{0\}.
\end{equation}
\end{subequations}
In deriving \eqref{dJ/dp} we used the fact that 
\begin{equation}
\label{expintegral}
\int_{\mathbb{T}^d}\!\! e^{i(\kb, \thetab)} \dm=\prod_{j=1}^d\int_{-\pi}^{\pi} e^{i k_j \theta_j} \frac{\derivd \theta_j}{2\pi}=0, \quad \kb\ne 0.
\end{equation} 
Therefore, if $\hat{P} \in \mathfrak{P}_{+, \circ}$ and $\hat{Q} \in \mathfrak{P}_{+}$, and hence the optimal solution is a stationary point of $\mathbb{J}$, then the spectrum $\Phi = \hat{P}/\hat{Q}$ fulfills both covariance matching \eqref{eq:Cov} and cepstral matching \eqref{eq:cepstrum}.

The following three lemmas ensure the existence of a solution and shows that the problem is in fact convex. The arguments are similar to those in the proof of Theorem~\ref{theo:conjecture}, and are given in 
the appendix.

\begin{lemma}\label{lem:JcontCeps}
Given $c \in \mathfrak{C}_+$ and a sequence $\gamma=\{\gamma_\kb\}_{\kb \in \Lambda}$ with $\gamma_{0}=1$ and $\gamma_{-\kb} = \bar{\gamma}_{\kb}$, the functional $(P,Q)\mapsto\mathbb{J}(P,Q)$ is lower semicontinuous on $\bar{\mathfrak{P}}_{+,\circ} \times (\bar{\mathfrak{P}}_+ \setminus \{0\})$.
\end{lemma}

\begin{lemma}\label{lm:jCompSublevel}
The sublevel sets $\mathbb{J}^{-1}(-\infty, r]$ are compact.
\end{lemma}

\begin{lemma}\label{lm:jCepsConv}
The dual problem (D) in Theorem~\ref{theo:cepstral} is convex on the domain $\bar{\mathfrak{P}}^{(n_1,\ldots,n_d)}_{+,\circ} \times \bar{\mathfrak{P}}^{(n_1,\ldots,n_d)}_+$.
\end{lemma}

Next we show that if $\hat{Q} \in \mathfrak{P}_{+}$ and $\hat{P} \in \mathfrak{P}_{+, \circ}$ then
$\hat\Phi = \hat{P}/\hat{Q}$ is also optimal for the primal problem of Theorem~\ref{theo:cepstral}. 
This  follows by observing that $\hat\Phi$ is a  primal feasible point and that the primal functional \eqref{cevcepstrprimal} takes the same values as the Lagrangian \eqref{eq:LagrangianCeps} in this point, since we have covariance and cepstral matching  (cf.\ the proof of Theorem \ref{theo:conjecture}). Finally, if  $d\leq 2$ then $\hat{Q} \in \mathfrak{P}_{+}$ whenever $\hat{P} \in \mathfrak{P}_{+, \circ}$, which follows directly from Lemma \ref{lem:noSolOnBoundary}.
This concludes the proof of Theorem~\ref{theo:cepstral}.
\end{proofWithName}

From this proof we see that the stationarity of $\mathbb{J}(P,Q)$ in $Q$ ensures covariance matching and the stationarity in $P$ provides cepstral matching. Therefore we can only guarantee matching for a solution in the interior $\mathfrak{P}_{+,\circ} \times \mathfrak{P}_+$. This subtle fact was overlooked in \cite{byrnes2002identifyability, enqvist2004aconvex}, where it is claimed that we also have covariance matching for $\hat{P} \in \partial \mathfrak{P}_{+,\circ}$. However, even when $d\leq 2$,  we cannot guarantee that there is a solution $\hat{Q}$ belonging  to the interior $\mathfrak{P}_+$ if $\hat{P} \in \partial \mathfrak{P}_{+,\circ}$. The following example illustrates this.

\begin{example}
Consider the one-dimensional problem with $c_0 = 2$, $c_{-1} = c_{1} = 1$ and $\gamma_1 = -1$. Set 
\begin{displaymath}
P(e^{i\theta}) = 1 - (e^{i\theta} + e^{-i\theta})/2 = 1 - \cos \theta, 
\end{displaymath}
and $Q = P$. Clearly $P$ and $Q$ belong to the boundary, since $P(e^{i0})=Q(e^{i0})=0$. Moreover $\Phi = P/Q = 1$, so there is neither covariance matching nor cepstral matching.  A simple calculation shows that 
$\partial \mathbb{J}/\partial q_0=\partial \mathbb{J}/\partial  q_1=\partial \mathbb{J}/\partial p_1=1$. 
However, for any feasible direction $(\delta q_0,\delta q_1,\delta p_1 )$ in $(P,Q)$ we have $\text{\rm Re}\{\delta p_1\}\ge 0$ and $\text{\rm Re}\{\delta q_0+2\delta q_1\}\ge 0$, and hence there is no feasible descent direction from this point. 
Therefore we have a local minimum, which, by convexity,  is also a global minimum.  Consequently, we have an optimal solution on the boundary where we have neither covariance nor cepstral matching.
\end{example}

\begin{remark}
From Theorem \ref{theo:conjecture} we know that it is possible to achieve covariance matching in this example by adding  a nonnegative singular measure $d\hat{\mu}$, representing spectral lines.  In fact, a similar statement can be proved for cepstral matching, namely that that there exists a nonpositive measure $d \tilde{\mu}$ such that $\supp (d \tilde{\mu}) \subseteq \{ \thetab \in \mathbb{T}^d \mid \hat{P}(\thetab) = 0 \}$ and
\begin{equation*}
\gamma_\kb = \int_{\mathbb{T}^d} e^{i (\kb,\thetab)} \left( \log (\hat{P}/\hat{Q}) \dm(\thetab)  - d\tilde{\mu}(\thetab) \right)
\end{equation*}
for all $\kb\in\Lambda \setminus \{0\}$. However, while the physical interpretation of $d \hat{\mu}$ in Theorem \ref{theo:conjecture} is clear, in this case it is not obvious what $d\tilde{\mu}$ represents in terms of the spectrum.
\end{remark}

Note that the optimization problem is convex but in general not strictly convex, and hence the solution might not be unique. This is illustrated in  the following example \cite[p. 504]{LindquistPicci2015}.

\begin{example}
Again consider a one-dimensional problem, this time with $c_0 = 1$, $c_{-1} = c_{1} = 0$ and $\gamma_1 = 0$. Choosing 
\begin{displaymath}
P(e^{i\theta}) = Q(e^{i\theta}) = 1 - \rho \cos\theta, \quad |\rho| \leq 1,
\end{displaymath}
we obtain  $\Phi = 1$, which matches the given  covariances and cepstral coefficients. Therefore all $P$ and $Q$ of this form are stationary points of $\mathbb{J}$  and are thus optimal for the dual problem in Theorem~\ref{theo:cepstral}.
\end{example}

In one dimension there is strict convexity, and thus a unique solution, if and only if there is an optimal solution for which $\hat{P}$ and $\hat{Q}$ are co-prime \cite{byrnes2002identifyability}.

\subsection{Regularizing the problem}\label{subsec:reg}
A motivation for simultaneous covariance and cepstral matching is to obtain a rational spectrum $\Phi = P/Q$ that matches the covariances without having to provide a prior $P$. However, even if $d\leq 2$, the dual problem in Theorem~\ref{theo:cepstral} cannot be guaranteed to produce such a spectrum that satisfies the covariance constraints \eqref{eq:Cov}. To remedy this we consider the regularization  proposed by Enqvist \cite{enqvist2004aconvex}, which has the objective function
\begin{displaymath}
\mathbb{J}_{\lambda}(P,Q) =\mathbb{J}(P,Q) - \lambda \int_{\mathbb{T}^d} \log P\, \dm ,
\end{displaymath}
where $\lambda \in (0, \infty)$ is the regularization parameter. 

The partial derivative with respect to $\bar{q}_\kb$ is identical to \eqref{dJ/dq}, whereas the partial derivative with respect to $\bar{p}_\kb$ becomes  
\begin{equation*}
\frac{\partial \mathbb{J}_\lambda(P,Q)}{\partial \bar{p}_\kb}  = \int_{\mathbb{T}^d} e^{i (\kb,\thetab)}\left( \log \left( \frac{P}{Q} \right) -\frac{\lambda}{P}\right)\dm - \gamma_\kb.
\end{equation*}
By Assumption \ref{ass:divergingIntegral}, this gradient  will be infinite for $P \in \partial \mathfrak{P}_+$, and hence the optimal solution is not on the boundary.  Moreover, with this regularization, the optimization problem becomes strictly convex and we thus have a unique solution.

\begin{theorem}\label{theo:cepstralReg}
Suppose that $d\leq 2$, and let $\gamma_\kb$,  $\kb \in \Lambda\setminus\{0\}$, be any sequence of complex numbers such that $\gamma_{-\kb} = \bar{\gamma}_\kb$. Set  $\gamma=\{\gamma_\kb\}_{\kb \in \Lambda}$ where $\gamma_0=1$, and let $c \in \mathfrak{C}_+$. Then for any $\lambda > 0$ there exists a unique solution $(\hat{P}, \hat{Q})$ to the strictly convex optimization problem to minimize 
\begin{equation*}
\mathbb{J}_{\lambda}(P,Q) = \langle c, q \rangle - \langle \gamma, p \rangle + \int_{\mathbb{T}^d} P \log \left(\frac{P}{Q}\right) \dm - \lambda \int_{\mathbb{T}^d} \log P\, \dm
\end{equation*}
subject to $P \in \mathfrak{P}_{+,\circ}$ and $Q \in \mathfrak{P}_+$. Moreover, $\Phi = \hat{P}/\hat{Q}$  fulfills the covariance matching \eqref{eq:Cov} and approximately fulfills the cepstral matching \eqref{eq:cepstrum} via
\begin{equation*}
\gamma_\kb + \varepsilon_\kb = \int_{\mathbb{T}^d} e^{i (\kb,\thetab)} \log \Phi \,\dm,\quad\text{where $\varepsilon_\kb = \lambda \int_{\mathbb{T}^d} e^{i(\kb, \thetab)} \hat{P}^{-1} \dm$}.
\end{equation*}
\end{theorem}

\begin{proof}
In view of what has been said, all of the results follow from Theorem \ref{theo:cepstral}  except the strict convexity. To prove this, we note that the second directional derivative of $\mathbb{J}_\lambda$ is given by 
\begin{equation*}
\delta^2 \mathbb{J}_\lambda (P, Q; \delta P, \delta Q) = \int_{\mathbb{T}^d}\!\! P \left( \delta P \frac{1}{P} -\delta Q \frac{1}{Q} \right)^2\dm + \int_{\mathbb{T}^d}\!\! \delta P^2 \frac{\lambda}{P^2} \dm
\end{equation*}
(cf. the proof of Lemma \ref{lm:jCepsConv} in the appendix). 
Since both integrands are nonnegative, they both need to be zero almost everywhere in order for the derivative to vanish. However, since $P > 0$, this implies that  $\delta P \equiv 0$ by continuity. Then the first integrand becomes $\delta Q^2 P/Q^2$ and in the same way we must thus have $\delta Q \equiv 0$.  Hence $\delta^2 \mathbb{J}_\lambda (P, Q; \delta P, \delta Q)>0$, implying uniqueness. 
\end{proof}

\section{The circulant problem}\label{sec:periodic}

Theorem \ref{theo:conjectureDisc} in Section \ref{subsec:periodProb} can be viewed as a periodic version of Theorem \ref{theo:conjecture} and Corollary \ref{theo:langMcclellan}, as can be seen by following the lines of  \cite{lindquist2013thecirculant}, where the one-dimensional problem was first introduced. To this end, we introduce the discrete measure $\derivd \nu_\Nb$, i.e.,
\begin{equation}
\label{eq:discreteMeasure}
\derivd \nu_\Nb (\thetab) =\sum_{\lb \in \mathbb{Z}^d_\Nb} \delta\big(\theta_1 -\phi_1(\ell_1), \ldots , \theta_d -\phi_d(\ell_d)\big) \prod_{j=1}^d\frac{\derivd \theta_j}{N_j} ,
\end{equation}
where $\phi_j(\ell):=2\pi\ell/N_j$ and $\delta$ is the multidimensional Dirac-delta function.  Then the moment matching condition \eqref{discretecov} takes the form
\begin{equation*}
c_\kb =  \frac{1}{\prod_{j=1}^d N_j} \sum_{\lb \in \mathbb{Z}_\Nb^d} \zetab_\lb^\kb \Phi(\zetab_\lb) = \int_{\mathbb{T}^d} e^{i (\kb,\thetab)} \Phi(e^{i\thetab}) d\nu_\Nb,
\end{equation*}
which is similar to \eqref{eq:Cov}, but where $d \nu_\Nb$ and $\dm$ have different mass distributions (discrete versus continuous). In fact, the main difference in the statements of Theorem \ref{theo:conjectureDisc}  and Theorem \ref{theo:conjecture} together with Corollary \ref{theo:langMcclellan}  is that different measures and cones are used. In the same way, versions of Theorems \ref{theo:cepstral} and \ref{theo:cepstralReg} also hold in the circulant case; see \cite{ringh2015themultidimensional} for details.

In connection to this it is also interesting to observe that the discrete counterpart of Assumption~\ref{ass:divergingIntegral},
\begin{equation}
\label{discreteass}
\int_{\mathbb{T}^d} \frac{1}{Q}d\nu_\Nb  = \infty \quad \text{for all  $Q \in \partial \mathfrak{P}_+(\Nb)$},
\end{equation}
 holds for any measure $d \nu_\Nb$ with discrete mass distribution 
(see also \cite{lang1982multidimensional}). However, if $P \in \partial \mathfrak{P}_+(\Nb)$ we may still obtain solutions without covariance matching, 
 because for any $Q$ that is zero only in a subset of points where $P$ is zero we will have $\int_{\mathbb{T}^d} (P/Q)\derivd \nu_\Nb < \infty$ and hence the optimal solution may occur on the boundary.

\begin{remark}\;
 Although the measure \eqref{eq:discreteMeasure} has mass in points placed in the roots of unity on the d-dimensional torus, one could also consider other mass distributions. One could place the mass points in the odd points of the roots of unity, i.e., in the points $\{e^{i (2k_j - 1) \pi/N_\ell}\}_{k_j=1}^{N_\ell}$, a situation which has been studied in the one-dimensional case and which correspond to spectra of skew-periodic processes \cite{ringh2014spectral}. The same holds in the multidimensional setting. Also note that all dimensions does not need to have mass distributions of the same type.
For example, the approach in this paper works even if the process is periodic in some of the dimensions, while non-periodic in others.
\end{remark}


\subsection{Convergence of discrete to continuous}\label{sec:convAnalys}
In \cite{lindquist2013thecirculant} Lindquist and Picci proved for the one-dimensional case that  when the number of mass points in the discrete measure $d\nu_\Nb$ in \eqref{eq:discreteMeasure} goes to infinity, the solution converges to the solution of the problem with the continuous measure $\dm$. The same is true in higher dimensions, and the formal result is given in Theorem \ref{theo:convergence} in Section \ref{subsec:periodProb}. In this subsection we will prove this statement. Note that  we use the notation
\begin{subequations}
\begin{align}
\mathbb{J}_P (Q) = \langle c, q \rangle - \int_{\mathbb{T}^d} P \log Q\, \dm \label{eq:JcontConv} \\
\mathbb{J}_P^{\Nb} (Q) = \langle c, q \rangle - \int_{\mathbb{T}^d} P \log Q \, d\nu_{\Nb} \label{eq:JdiscConv}
\end{align}
\end{subequations}
to explicitly distinguish the objective functions using the continuous and the discrete measure.
Moreover let $\hat{Q}$ be the minimizer of \eqref{eq:JcontConv}, subject to $Q \in \bar{\mathfrak{P}}_+$, and $\hat{Q}_\Nb$ be a minimizer of \eqref{eq:JdiscConv}, subject to $Q \in \bar{\mathfrak{P}}_+(\Nb)$. Before proving the theorem, we make some clarifying observations.


\begin{remark}
We have already noted that the singular measure $d\hat{\mu}$ is not unique. However, the corresponding ``rest covariance'' $\hat{c}$, which $d\hat{\mu}$ matches, is unique (cf. equation \eqref{eq:cHat}). In connection to this it is interesting to note that although this is the case, and although $\hat{Q}_\Nb \rightarrow \hat{Q}$, in general $\hat{c}_\Nb \not \rightarrow \hat{c}$. To see this, note that for a $P$ which is positive in all points except for some irrational frequency%
\footnote{
An irrational frequency is an angle $\lambda\pi$ for which $\lambda$ is an irrational number.
} 
where $P= 0$, we will have $P \in \mathfrak{P}_+(\Nb)$ for all $\Nb$, since this point will never belong to the grid. Thus we will have $\hat{Q}_\Nb \in \mathfrak{P}_+(\Nb)$ and therefore $\hat{c}_N = 0$. However $P \in \partial \mathfrak{P}_+$, and therefore we can have $\hat{Q} \in \partial \mathfrak{P}_+$ and hence $\hat{c} \not = 0$. One can construct such example based on Example~\ref{ex:1Dexamplesingular} by shifting the spectral line to an irrational frequency point.  
\end{remark}

\begin{remark}
In connection to the previous remark, we note that in two dimensions we have $\hat{Q} \in \mathfrak{P}_+$ whenever $P \not \in \partial\mathfrak{P}_+$, since Assumption~\ref{ass:divergingIntegral} is valid for $d=2$. Hence there will be no singular measure. Moreover,  since $\hat Q_\Nb\to\hat Q$ as $\min(\Nb)$ goes to infinity, for large enough value of $\min(\Nb)$ we must have $\hat{Q}_\Nb > 0$, i.e., $\hat{Q}_\Nb \in \mathfrak{P}_+$. Therefore $(P/\hat{Q}_\Nb) d \nu_\Nb$ tends to $(P/\hat{Q}) \dm$ in weak$^*$.
\end{remark}

The first thing we need to show is that $\hat{Q}_\Nb$ is in fact well-defined. That this is not evident from the statement of the theorem becomes apparent when noting the following relationship among the cones of trigonometric polynomials:
\begin{equation*}
\bar{\mathfrak{P}}_+(\Nb) \supset \bar{\mathfrak{P}}_+(2\Nb) \supset \ldots \supset \bar{\mathfrak{P}}_+.
\end{equation*}
For the dual cones we therefore have \cite[pp. 157-158]{luenberger1969optimization}
\begin{equation*}
\bar{\mathfrak{C}}_+(\Nb) \subset \bar{\mathfrak{C}}_+(2\Nb) \subset \ldots \subset \bar{\mathfrak{C}}_+,
\end{equation*}
and thus it is not guaranteed that minimizing \eqref{eq:JdiscConv} over $Q \in \bar{\mathfrak{P}}_+(\Nb)$ has a solution for $c \in \mathfrak{C}_+$.
However note that when $N_l \rightarrow \infty$ the corresponding set $\{ e^{ik_l 2\pi/N_l} \}_{k_l \in \mathbb{Z}_{N_l}}$ will become dense on the unit circle. Therefore $\bar{\mathfrak{P}}_+ = \bigcap_{\Nb \in \mZ^d_+} \bar{\mathfrak{P}}_+(\Nb)$. Using this we have the following  
lemma, proved in the appendix,  which is a 
generalization to the multivariable case of Proposition~6 in \cite{lindquist2013thecirculant}.

\begin{lemma}\label{lm:limN}
For any $c \in \mathfrak{C}_+$ there exist an $N_0$ such that $c \in \mathfrak{C}_+(\Nb)$ for all $\min(\Nb) \geq N_0$.
\end{lemma}

This shows that for each $c \in \CP$, the problem of minimizing \eqref{eq:JdiscConv} over $Q \in \bar{\mathfrak{P}}_+(\Nb)$ does in fact have a solution for large enough values of $\Nb$. Interestingly, the lemma is equivalent to $\lim_{\min(\Nb) \to \infty} \CP(\Nb)=\CP$.

\begin{proofWithName}{Proof of Theorem \ref{theo:convergence}}
Let $\hat{Q}$ and $\hat{Q}_\Nb$ be as in the statement of the theorem. Choose a $c \in \mathfrak{C}_+$ and a $P \in \bar{\mathfrak{P}}_+ \setminus \{ 0\}$ and fix $N_0$ in accordance with Lemma \ref{lm:limN}. Throughout the rest of this proof we only consider $\min(\Nb) \geq N_0$, which means that an optimal solution $\hat{Q}_\Nb$ exists. Moreover, in the proof we need the following result, which is proved in 
the appendix.

\begin{lemma}\label{lem:seqBounded}
The sequence $(\hat{Q}_\Nb)$ is bounded in $\LInf$.
\end{lemma}

Since  $(\hat{Q}_{\Nb})$ is bounded, there is a convergent subsequence, call it $(\hat{Q}_{\Nb})$ for convenience, converging in the $\LInf$ norm to some function $\hat{Q}_{\infty}$. Since $(\hat{Q}_{\Nb})$ is a set of continuous functions, this means that the convergence is in fact uniform and hence $\hat{Q}_{\infty}$ is a continuous function. Now since i) the convergence is uniform, ii) $\hat{Q}_{\infty}$ is continuous, and iii) the grid points become dense on $\mT^d$ as $\min( \Nb )$ goes to infinity, we obtain $\hat{Q}_{\infty}(e^{i\thetab}) \geq 0$ for all $\thetab$, and hence $\hat{Q}_{\infty}$ belongs to $\bar{\mathfrak{P}}_+  \setminus \{ 0 \}$.

It remains to show that  $\hat{Q}_{\infty} = \hat{Q}$. This will be done by proving that  $\|\hat{Q}_\infty - \hat{Q}\|_\infty \leq \varepsilon$ for all $\varepsilon > 0$. To do this, fix a $\tilde{Q} \in \mathfrak{P}_+$ and consider $\hat{Q} + \eta \tilde{Q}$, which belongs to $\mathfrak{P}_+$ for all $\eta > 0$. By simply adding and subtracting $\eta \tilde{Q}$, the triangle inequality gives 
\begin{equation}\label{eq:convProof1}
\| \hat{Q}_\infty  - \hat{Q} \|_\infty \leq \eta \|\tilde{Q}\|_\infty + \| (\hat{Q}_\infty + \eta \tilde{Q}) - \hat{Q} \|_\infty.
\end{equation}
We want to bound the second term. To this end, note that
\begin{equation*}
\mathbb{J}_P (\hat{Q} + \eta \tilde{Q}) - \mathbb{J}_P (\hat{Q}) = \langle c, \eta \tilde{q} \rangle - \int_{\mathbb{T}^d}\!\! P \log \left( 1 + \frac{\eta \tilde{Q}}{\hat{Q}} \right) \dm ,
\end{equation*}
and, since the integral is nonnegative, we obtain
\begin{equation}\label{J(Qhat+Delta)}
\mathbb{J}_P (\hat{Q} + \eta \tilde{Q}) \leq \mathbb{J}_P (\hat{Q}) + \eta \langle c, \tilde{q} \rangle.
\end{equation}
The same holds for $\mathbb{J}_P^\Nb$, i.e., $\mathbb{J}_P^\Nb(\hat{Q}_\Nb  + \eta \tilde{Q}) \leq \mathbb{J}_P^\Nb(\hat{Q}_\Nb) + \eta \langle c, \tilde{q} \rangle$.
By optimality we also have $\mathbb{J}_P^\Nb(\hat{Q}_\Nb) \leq \mathbb{J}_P^\Nb(\hat{Q}  + \eta \tilde{Q}) < \infty$ for all $\eta > 0$, and hence
\begin{equation}\label{eq:convAnal2}
 \mathbb{J}_P^\Nb(\hat{Q}_\Nb  + \eta \tilde{Q}) \leq \mathbb{J}_P^\Nb(\hat{Q}  + \eta \tilde{Q})  + \eta \langle c, \tilde{q} \rangle.
\end{equation}
Now, since $\hat{Q}_\Nb  + \eta \tilde{Q} \rightarrow \hat{Q}_{\infty} + \eta \tilde{Q} \in \mathfrak{P}_+$, we know that, for large enough values of $\min (\Nb)$, we have $\hat{Q}_\Nb  + \eta \tilde{Q} \in \mathfrak{P}_+$. Therefore, the left hand side of \eqref{eq:convAnal2} is guaranteed to be well-defined for all values of $\min (\Nb)$§ larger than this value. We can thus take the limit on both sides of  
\eqref{eq:convAnal2} to obtain
\begin{equation*}
\mathbb{J}_P(\hat{Q}_\infty  + \eta \tilde{Q}) \leq \mathbb{J}_P(\hat{Q}  + \eta \tilde{Q}) + \eta \langle c, \tilde{q} \rangle,
\end{equation*}
which together with \eqref{J(Qhat+Delta)} yields
\begin{equation}
\label{eq:convAnal3}
\mathbb{J}_P(\hat{Q}_\infty  + \eta \tilde{Q}) \leq  \mathbb{J}_P (\hat{Q}) + 2\eta \langle c, \tilde{q} \rangle.
\end{equation}
Now consider the sets  $D_\delta = \{ Q \in \bar{\mathfrak{P}}_+ \; | \; \mathbb{J}_P(Q) \leq \mathbb{J}_P(\hat{Q}) + \delta \}$.
Since the Hessian at the optimal solution is positive definite we have $\bigcap_{\delta > 0} D_\delta = \{ \hat{Q} \}$. Therefore, it follows from \eqref{eq:convAnal3} that $\eta > 0$ can be chosen so that $\| (\hat{Q}_\infty  + \eta \tilde{Q}) - \hat{Q} \|_\infty < \tilde{\varepsilon}$ for any $\tilde{\varepsilon} > 0$. Consequently, by selecting $\eta$ sufficiently small, we may bound \eqref{eq:convProof1} by an arbitrary small positive number. Hence $\hat{Q}_\infty = \hat{Q}$. 
%
%
\end{proofWithName}

\section{Application to system identification}\label{sec:sysIdEx}
The power spectrum of a signal represents the energy distribution across frequencies of the signal. For a multidimensional, discrete-time, zero-mean, and homogeneous%
\footnote{
Homogeneity  implies that covariances $c_\kb:=E\{y({\bf t}+\kb)\overline{y({\bf t})}\}$ are invariant with ``time'' ${\bf t}\in \mZ^d$. From this it is also easy to see that $c_{-\kb} = \bar{c}_\kb$.}
stochastic process  $\{y({\bf t})\}$, defined for ${\bf t}\in\mZ^d$, the power spectrum is defined as the nonnegative measure $d\mu$ on $\mT^d$ whose Fourier coefficients are the covariances
\begin{equation*}
c_\kb = \int_{\mT^d} e^{i(\kb,\thetab)} d\mu.
\end{equation*}
In one dimension the singular part of the measure represents spectral lines, and if the absolutely continuous part is also rational, $\Phi = P/Q$, one can use spectral factorization to determine the filter coefficients for an autoregressive-moving-average (ARMA) model which, when feed with white noise input, reproduces a stochastic signal with the same power distribution as $\Phi$. Therefore the one-dimensional rational covariance extension problem can be used for system identification \cite{LindquistPicci2015}. 

With the theory developed in this paper we can estimate rational spectra in higher dimensions. However spectral factorization is not in general possible when $d > 1$  \cite{dumitrescu2007positive}. For $d = 2$, Geronimo and Woerdeman have established conditions for when it is possible to factorize a given trigonometric polynomial as a sum-of-one-square \cite[Thm. 1.1.1]{geronimo2004positive}. These includes a non-trivial rank condition on a reduced matrix of Fourier coefficients, which we shall call $\Gamma_{\text{red}}$,  but also gives an explicit algorithm for obtaining the factors in cases when it is possible. Nevertheless, in the following example we will illustrate how the theory could be used in the case when covariances and cepstral coefficients comes from a rational, factorizable spectrum.

We consider a 2D recursive filter with transfer function
\[
\frac{b(e^{i\theta_1}, e^{i\theta_2})}{a(e^{i\theta_1}, e^{i\theta_2})} = \frac{ \sum_{\kb \in \Lambda_+} b_{\kb} e^{-i(\kb,\thetab)}}{\sum_{\kb \in \Lambda_+} a_{\kb} e^{-i(\kb,\thetab)}},
\]
where $\Lambda_+=\{(k_1,k_2)\in \mZ^2\mid 0\le k_1\le 2, 0\le k_2\le 2 \}$ and the coefficients are given by
$b_{(k_1,k_2)}=B_{k_1+1, k_2+1} $ and $a_{(k_1,k_2)}=A_{k_1+1, k_2+1}$, where
\[
B =
\begin{bmatrix}
    0.9589 &  -0.0479 &   0.0959 \\
    0.0959 &   0.0479 &   0.0959 \\
   -0.0959 &   0.0479 &   0.1918
\end{bmatrix},
%
%
\quad
A =
\begin{bmatrix}
    1.0000 &   0.1000 &   0.0500 \\
   -0.1000 &   0.0500 &  -0.0500 \\
    0.2000 &  -0.0500 &  -0.1000
\end{bmatrix} .
%
%
\]
Then the corresponding spectrum is given by
\[
\Phi(e^{i\thetab}) = \Phi(e^{i\theta_1}, e^{i\theta_2}) = \frac{P(e^{i\theta_1}, e^{i\theta_2})}{Q(e^{i\theta_1}, e^{i\theta_2})} = \left|\frac{b(e^{i\theta_1}, e^{i\theta_2})}{a(e^{i\theta_1}, e^{i\theta_2})}\right|^2,
\]
and hence the index set $\Lambda$ of the coefficients of the trigonometric polynomials $P$ and $Q$ is given by $\Lambda=\{(k_1,k_2)\in \mZ^2\,|\; |k_1|\le 2, |k_2|\le 2 \}$.

We approximate the continuous problem with a discrete one in accordance with Theorem \ref{theo:convergence}. The two-dimensional spectrum $\Phi$ is evaluated on a grid of size $30 \times 30$, and shown in Figure~\ref{fig:trueSpectrum}. The trigonometric polynomials corresponding to the true spectrum are shown in Figure~\ref{fig:truePQ}. Its covariances and cepstral coefficients are computed,
and a spectrum is then estimated by (unregularized) covariance and cepstral matching along the lines of Theorem \ref{theo:cepstral}. The problem is solved numerically using CVX, a Matlab package for solving disciplined convex programming problems \cite{cvx, grant2008graph}, and the resulting spectrum is shown in Figure~\ref{fig:estim_spectrum}. The relative error\footnote{Let the relative error between two functions $\Phi_{\text{true}}$ and $\Phi_{\text{est}}$ be the point-wise evaluation of $|\Phi_{\text{true}} - \Phi_{\text{est}}|/\Phi_{\text{true}}$.}
is shown in Figure~\ref{fig:estim_spectrumError}.  As seen from the relative error,  we recover the true spectrum with good accuracy.  For the ME solution, the resulting spectrum and relative error is shown in Figure~\ref{fig:meSpectrum}.

\begin{figure}%
  \centering
  \includegraphics[width=0.45\textwidth]{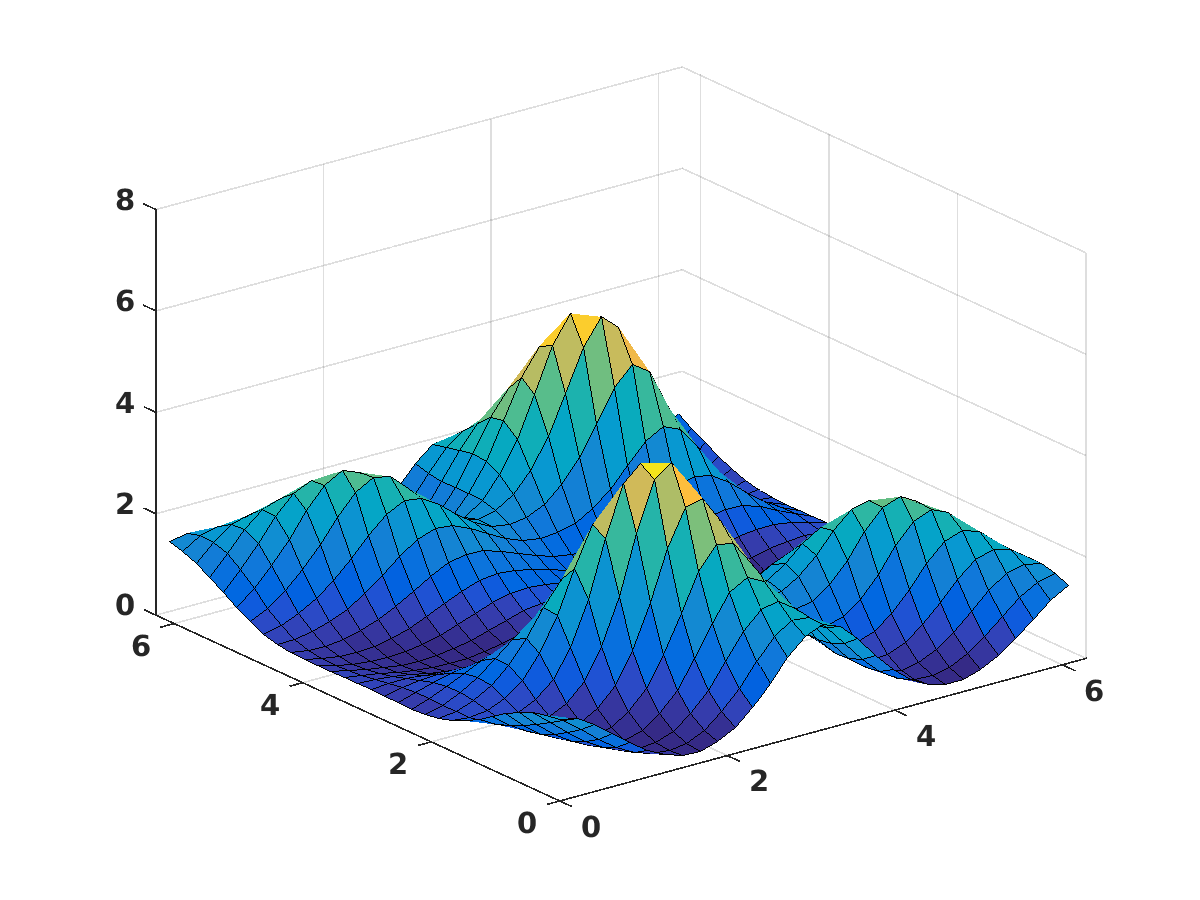}
  \caption{The true spectrum.}
  \label{fig:trueSpectrum}
\end{figure}

\begin{figure*}%
  \centering
  \hfil
  \subfloat[The true polynomial $P$. \label{fig:true_P}]{\includegraphics[width=0.45\textwidth]{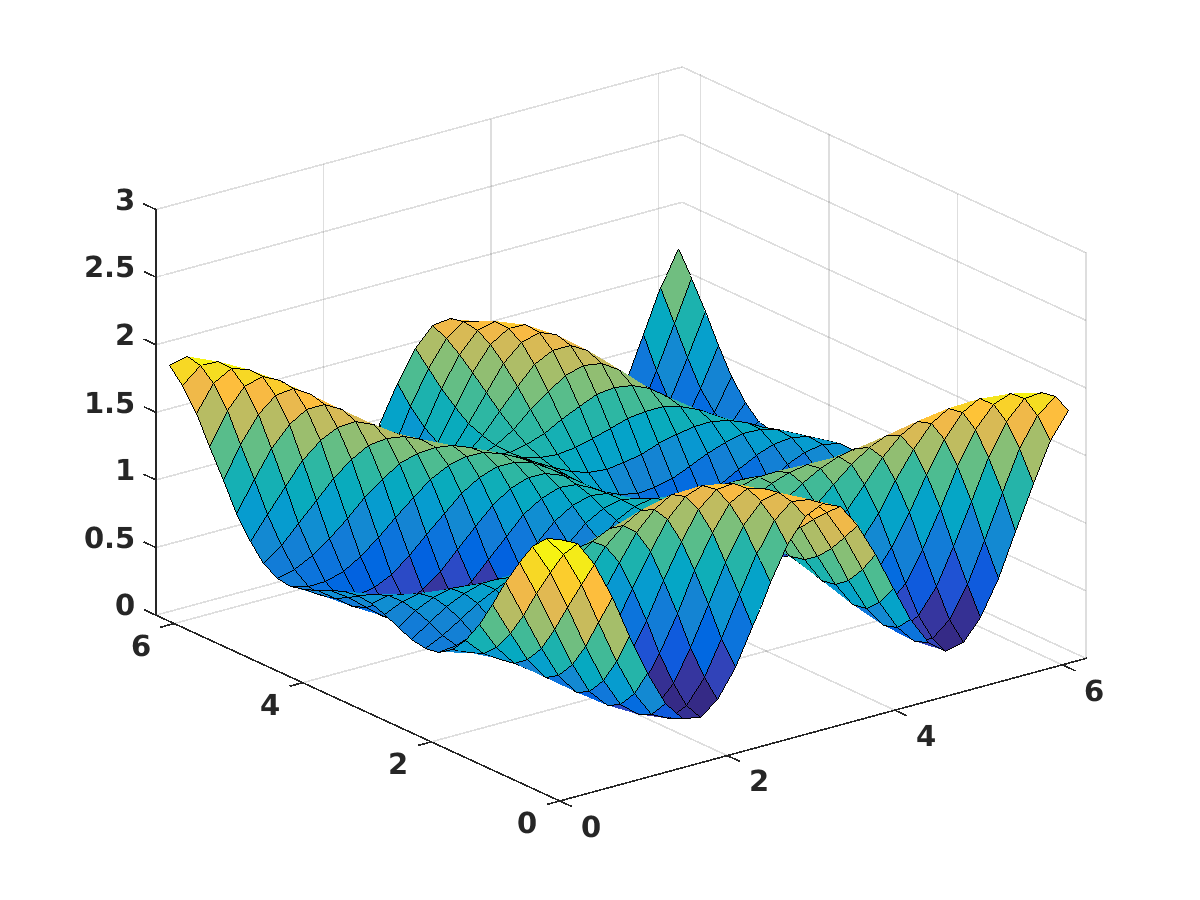}}%
  \hfil
  \subfloat[The true polynomial Q. \label{fig:true_Q}]{\includegraphics[width=0.45\textwidth]{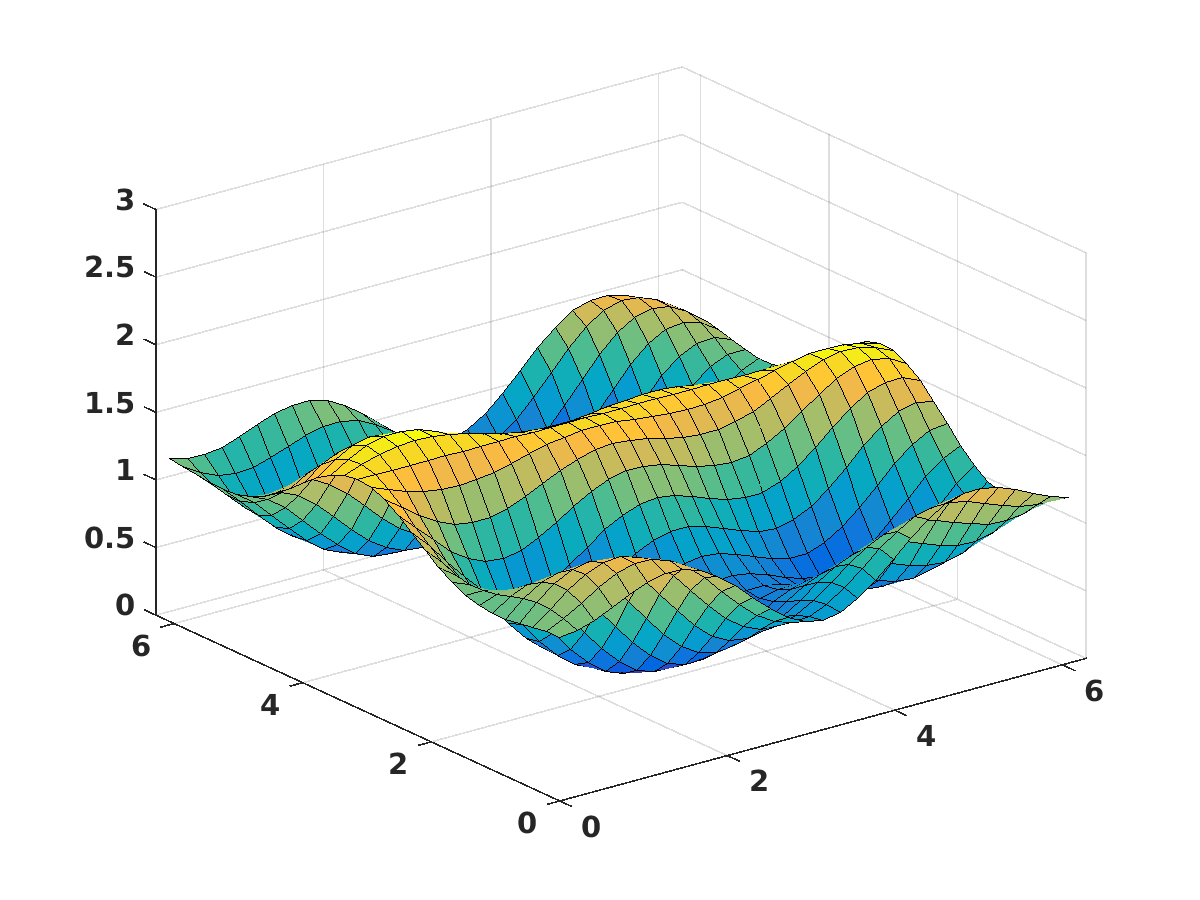}}%
  \caption{The spectrum of the system}%
  \label{fig:truePQ}%
\end{figure*}

\begin{figure*}%
  \centering
  \hfil
  \subfloat[Estimated spectrum. \label{fig:estim_spectrum}]{\includegraphics[width=0.45\textwidth]{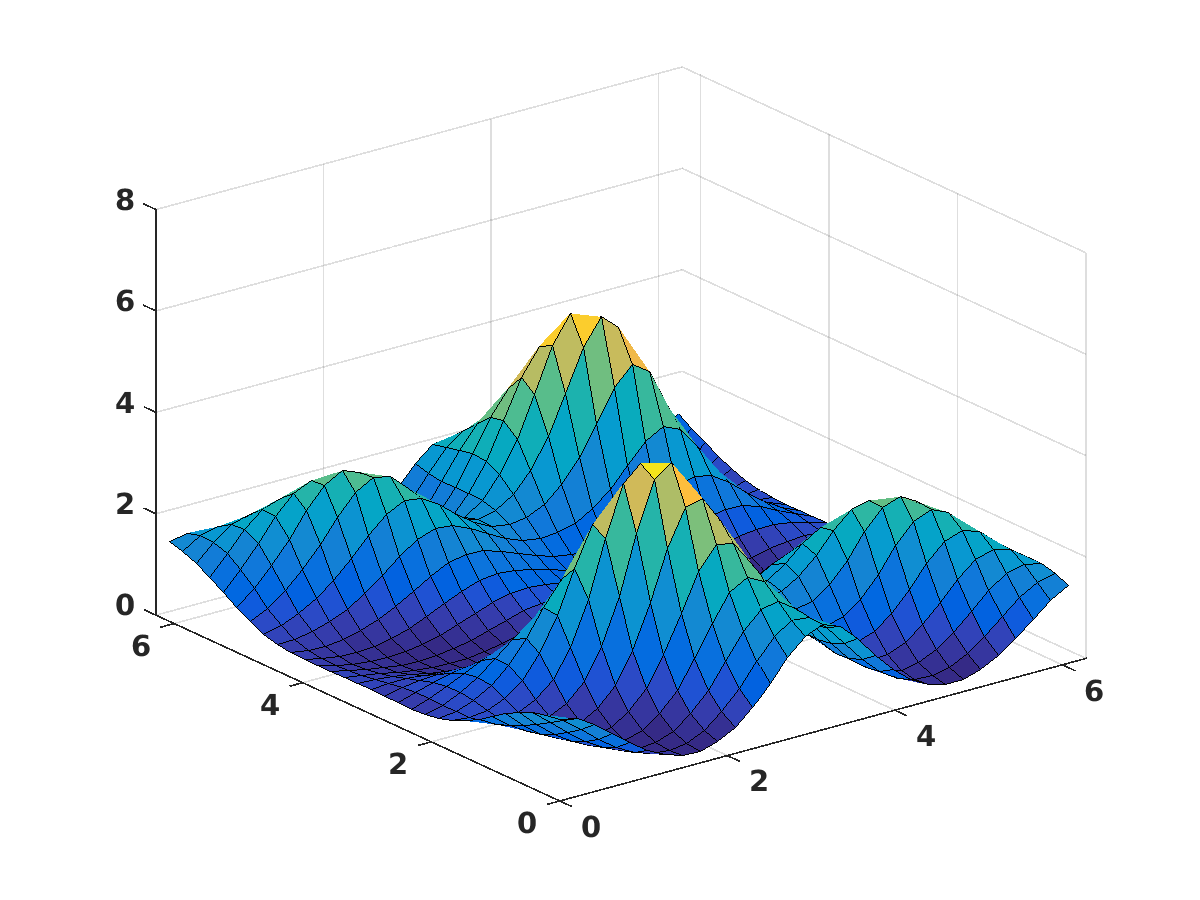}}%
  \hfil
  \subfloat[Relative error. \label{fig:estim_spectrumError}]{\includegraphics[width=0.45\textwidth]{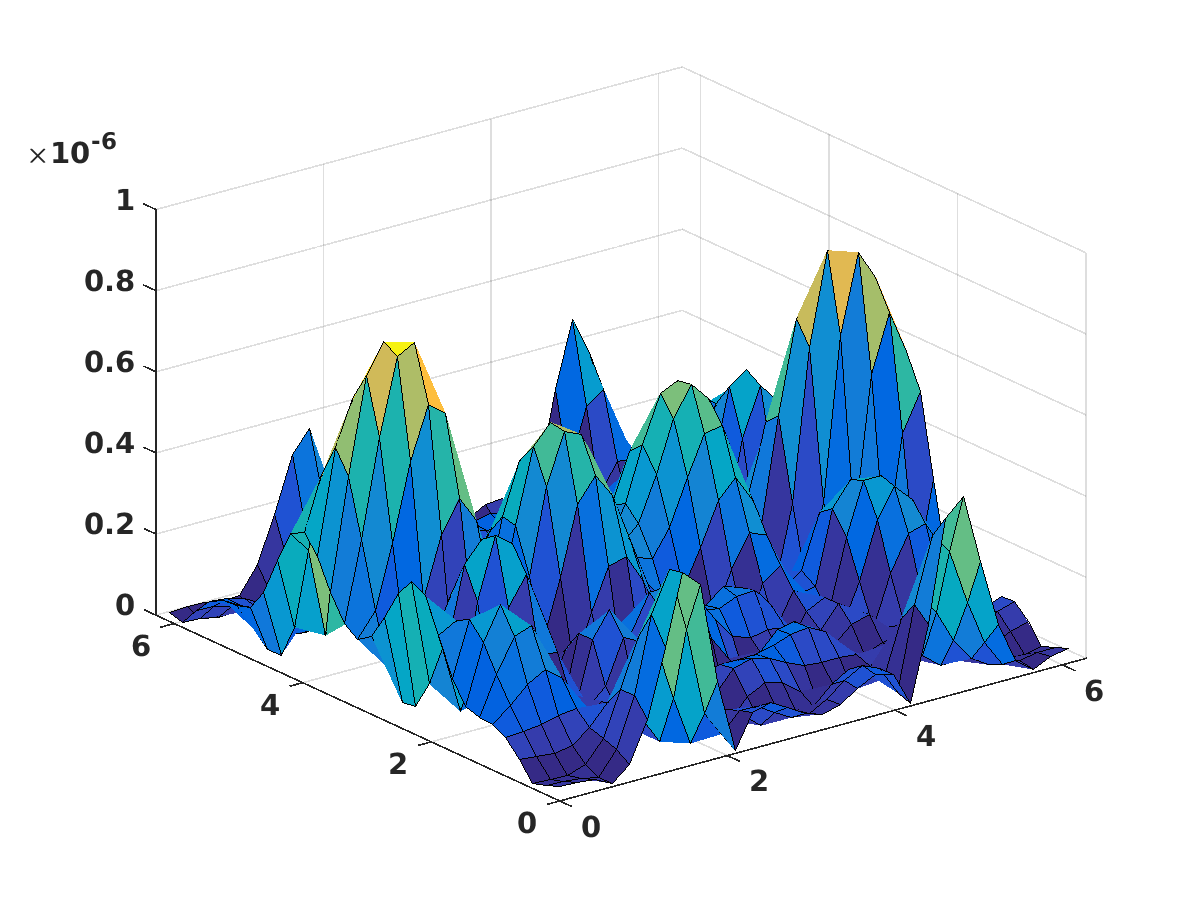}}%
  \caption{Spectrum estimated with covariance and cepstral matching.}%
  \label{fig:estimSpectrum}%
\end{figure*}

\begin{figure*}%
  \centering
  \hfil
  \subfloat[ME-spectrum. \label{fig:me_spectrum}]{\includegraphics[width=0.45\textwidth]{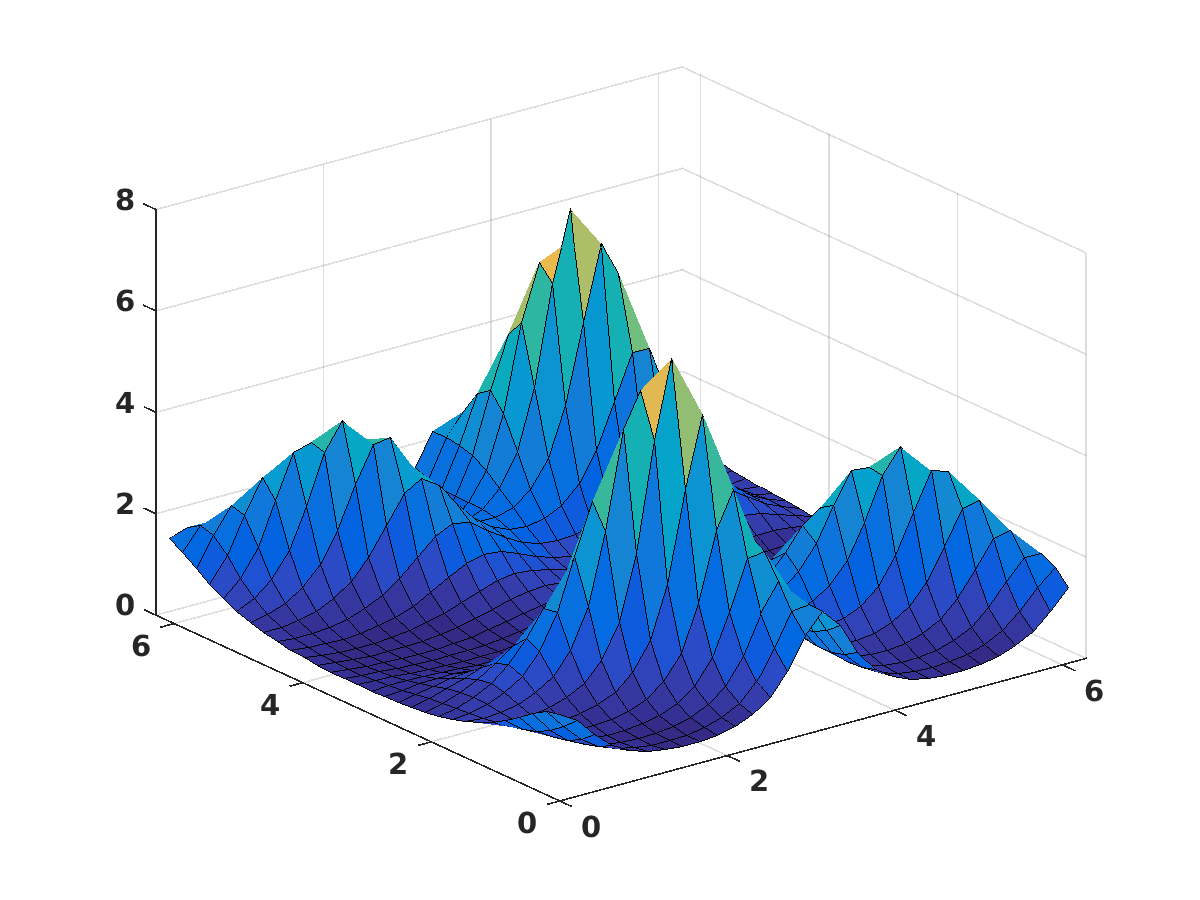}}%
  \hfil
  \subfloat[Relative error. \label{fig:me_spectrumError}]{\includegraphics[width=0.45\textwidth]{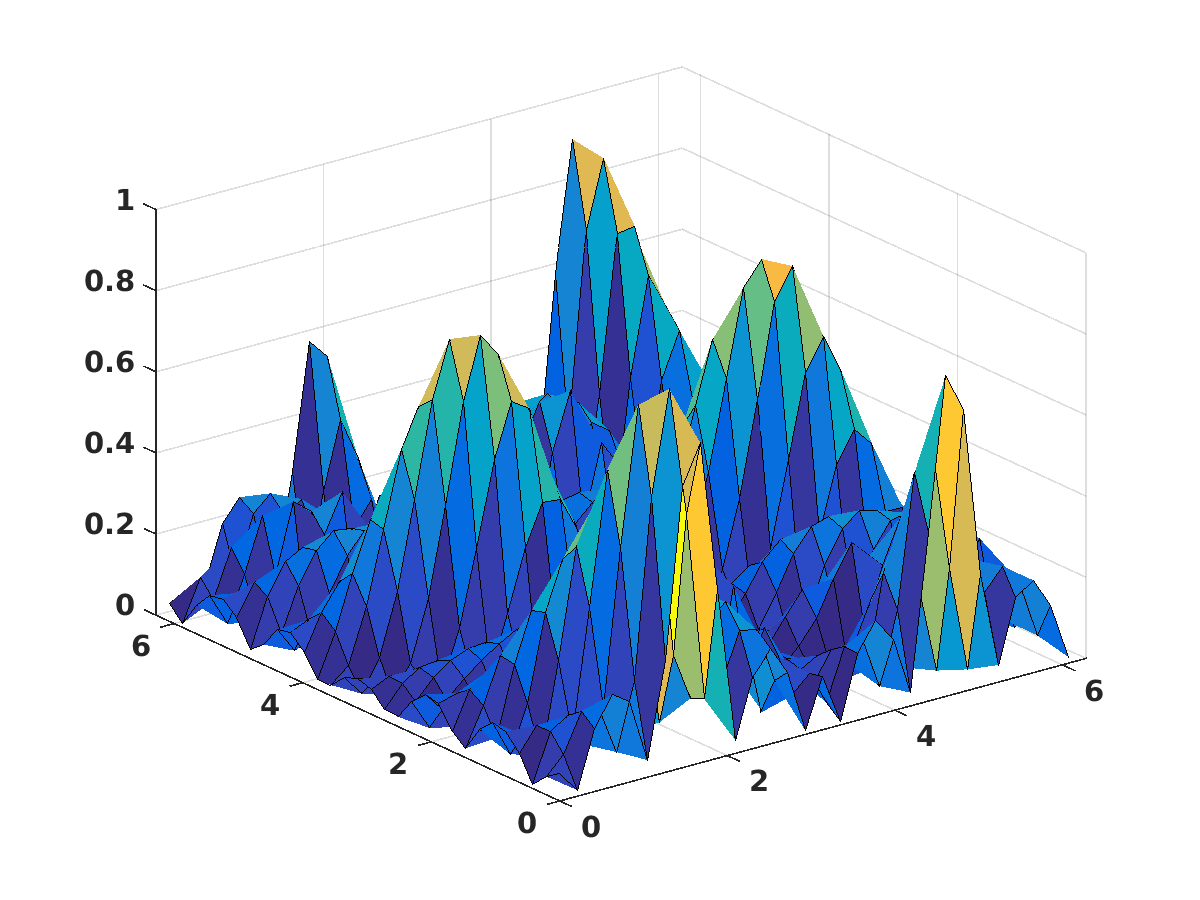}}%
  \caption{The ME-estimation and relative error to true spectrum.}%
  \label{fig:meSpectrum}%
\end{figure*}

For system identification we are now interested in factorizing the two rational spectra as a sum-of-one-square, if possible.
To check factorizability for the two solutions, we apply the rank condition from \cite[Theorem 1.1.1]{geronimo2004positive}, which requires that the corresponding submatrix $\Gamma_{\text{red}} \in \mathbb{C}^{6 \times 6}$ should be of rank four in both cases. 
However, such a matrix is generically full rank and we have to study the singular values in order to determine the numerical rank. 

To illustrate this issue, in Figure~\ref{fig:svd} we plot the singular values of $\Gamma_{\text{red}}$ for the respective polynomials. Figure~\ref{fig:svd_Q} shows  the singular values corresponding to the solution $Q_{\text{true P}}$ computed with the true polynomial $P$  as prior (cf. Theorem \ref{theo:conjecture} and Section \ref{subsec:commentsAndExamples}). This solution, as well as the solution obtained by covariance and cepstral matching, gives the exact spectrum back, up to numerical errors, and hence should be factorizable.  For both these solutions we can also observe a significant decrease in size between the fourth and the fifth singular values in Figure~\ref{fig:svd_Q}. This indicates that the matrices in fact have numerical rank four, and spectral factorization is thus possible.
Performing the spectral factorization on the solution with covariance and cepstral matching gives polynomials with coefficients
\begin{equation*}
B_{\text{est}} = 
\begin{bmatrix}
    0.9589 &  -0.0479 &   0.0959 \\
    0.0959 &   0.0479 &   0.0959 \\
   -0.0959 &   0.0479 &   0.1918
\end{bmatrix},
\quad
A_{\text{est}} =
\begin{bmatrix}
    1.0000 &   0.1000 &   0.0500 \\
   -0.1000 &   0.0500 &  -0.0500 \\
    0.2000 &  -0.0500 &  -0.1000
\end{bmatrix},
%
%
\end{equation*}
which agree completely with the true coefficients.

For the ME spectrum on the other hand there is no guarantee that it will be factorizable. In general there is {\em a priori\/} no reason why spectral factorization should be possible.
 However, 
in Figure~\ref{fig:svd_Q} we observe a decrease in size between the fourth and the fifth singular values also for the ME solution $\Phi_\text{ME}=1/Q_\text{ME}$, although this decrease is significantly smaller than for the other polynomials. If for the moment we assume that the rank condition on $\Gamma_\text{red}$ is actually (approximately) satisfied and apply the factorization algorithm of \cite{geronimo2004positive}, we obtain the coefficients 
\begin{equation*}
A_{\text{ME}} =
\begin{bmatrix}
    1.0317 &   0.1423 &  -0.0251 \\
   -0.1881 &  -0.0173 &  -0.1252 \\
    0.2872 &  -0.0570 &  -0.2597
\end{bmatrix}
\end{equation*}
for the possible spectral factor $a_\text{ME}$ of $Q_\text{ME}$. Forming the corresponding true $Q$, namely $|a_\text{ME}|^2$, and comparing it with $Q_\text{ME}$, we obtain a relative error of up to 10\% with respect to $Q_\text{ME}$. We leave the question whether this is a reasonable approximation to a future study.
Note also that if the ME spectrum is factorizable, the factors are given directly from the covariances by the Geronimo and Woerdeman algorithm. However if this is not the case, rational covariance extension will still give a rational spectrum.
An important open question related to this, and suggested by the above analysis, is whether the solution can be tuned by an appropriate choice of $P$ so that the rank condition is satisfied, and hence factorization is possible.

%
%

\begin{figure*}%
  \centering
  \hfil
  \subfloat[Singular values of $\Gamma_{\text{red}}$ for different $P$. \label{fig:svd_P}]{\includegraphics[width=0.45\textwidth]{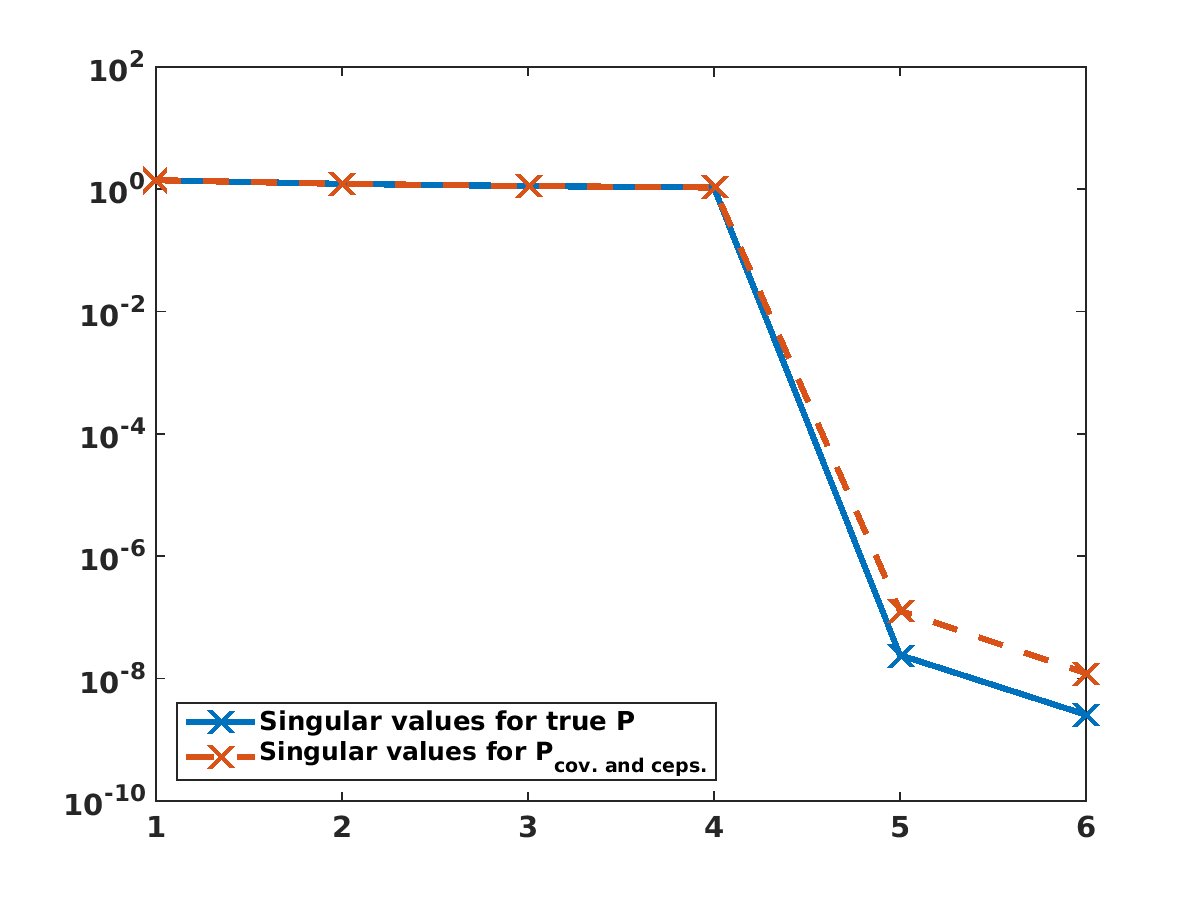}}%
  \hfil
  \subfloat[Singular values of $\Gamma_{\text{red}}$ for different $Q$. \label{fig:svd_Q}]{\includegraphics[width=0.45\textwidth]{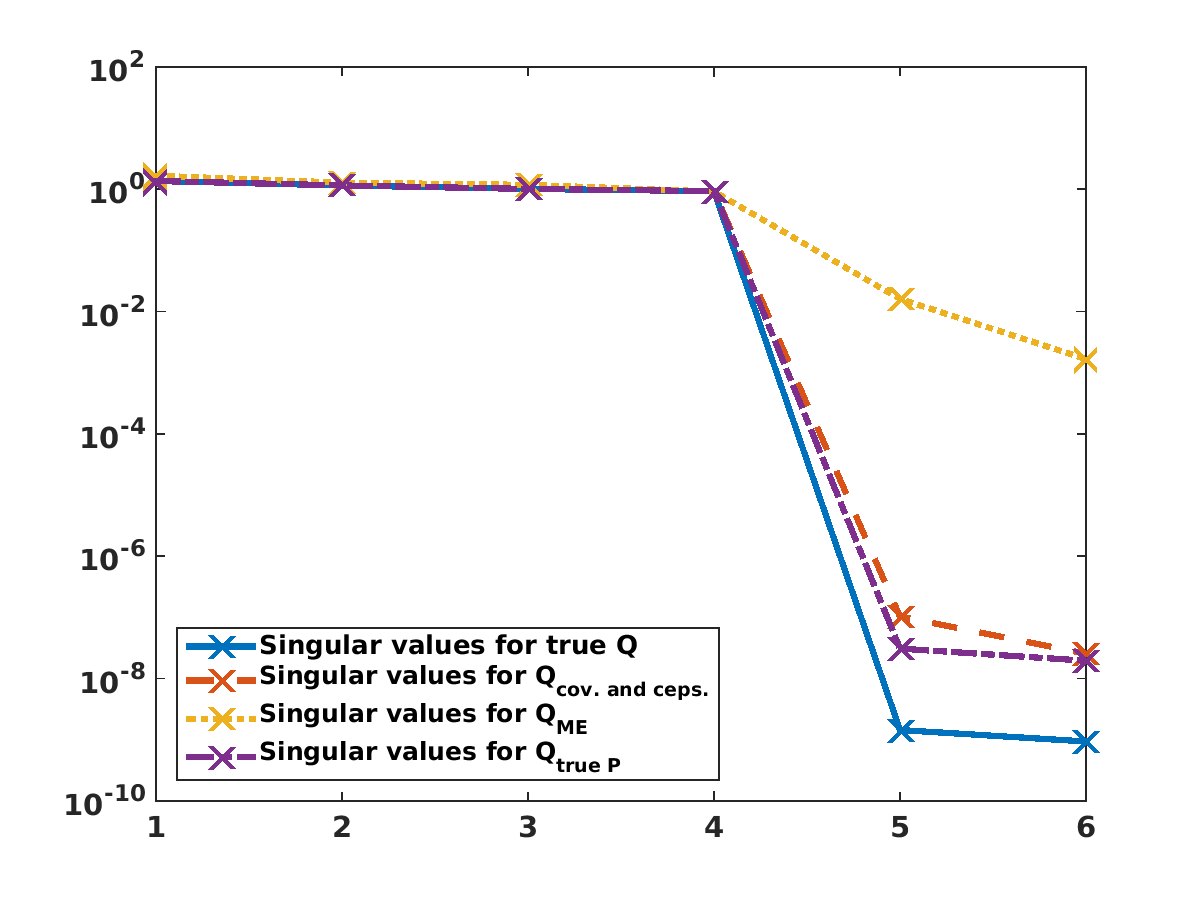}}%
  \caption{The singular values of the reduced covariance matrix.}%
  \label{fig:svd}%
\end{figure*}

\section{Application to image compression}\label{sec:imageComp}
Since the expression \eqref{eq:rat} is determined by a limited number of parameters, this approach enables compression of data. Moreover, the smoothness of the parameterization will facilitate tuning to specifications.
Therefore we apply the two-dimensional  circulant RCEP to  compression of black-and-white images. 
Compression is achieved by approximating the image with a rational spectrum, thereby using fewer parameters.
We compare the ME spectrum to the solution resulting from regularized covariance and cepstral matching. By choosing $n_1 \ll N_1$, $n_2 \ll N_2$, where $N_1$ and $N_2$ are the dimensions of the image, we obtain a significant reduction in number of parameters describing the image.

A seemingly straight-forward way is to compute the covariances and cepstral coefficients directly from the image, and then use these to compute the spectrum. However, if the discrete spectrum is zero in one of the grid points, the  (discrete) cepstrum is not well-defined. Hence simultaneous covariance and cepstral matching cannot be applied. Therefore we transform the image, denoted by $\Psi$, using $\Phi = e^{\Psi}.$ Since $\Psi$ is real, $\Phi$ is guaranteed to be real and positive for all discrete frequencies, and $\Psi$ is obtained as $\Psi = \log\Phi$. We then compute \eqref{eq:Cov} and \eqref{eq:cepstrum} and obtain the approximant $\hat \Phi$ from Theorem~\ref{theo:cepstralReg}. 
Here we use the real sequences of covariances and cepstral coefficients obtained by extending the image by symmetric mirroring (i.e., using the discrete cosine transform \cite[Section 4.2]{rao2014discrete}). However, the covariances and cepstral coefficients of $\Phi$ can also be computed as the inverse 2D-FFT of $e^{\Psi}$ and $\Psi$ respectively.

Moreover, note that a ME solution of the same maximum degree as a solution with a full-degree $P$  have about half the number of parameters. To compensate for this, we let the degree of the ME solution be a factor $\sqrt{2}$ higher (rounded up), in order to get a fair comparison.

\subsection{Compression of simplistic images}
To better understand the  different methods we first perform compression on a simple image of only black and white squares. The original image is shown in Figure~\ref{fig:pixel_true} and various results are shown in Figure~\ref{fig:pixelImages}.  Figure~\ref{fig:pixel_ceps_n4}, shows that, if too few coefficients are used, the compression cannot represent the harmonics present in the image, regardless of the use of a nontrivial $P$.
A visual assessment of the result shows that \ref{fig:pixel_me_n6} clearly outperforms \ref{fig:pixel_ceps_n4}, and that \ref{fig:pixel_me_n15} is still slightly better than \ref{fig:pixel_ceps_n10}. However \ref{fig:pixel_ceps_n11} and \ref{fig:pixel_ceps_n30} are better than \ref{fig:pixel_me_n16} and \ref{fig:pixel_me_n43}, respectively. In order to more objectively assess the quality of the two different compression methods, we also compute the MSSIM value of the compressed images. This is a measure, taking values in the interval $[0,1]$,  for evaluating quality and degradation of images, for which $1$ means exact agreement \cite{wang2004image}. A plot of the MSSIM value for compressions of different degree is shown in Figure~\ref{fig:MSSIM_graph}. However note that this measure does not agree completely with the visual impression of all images. Most notably, the measure gives a higher value to the grey image in Figure~\ref{fig:pixel_ceps_n4} than the image with structure in Figure~\ref{fig:pixel_me_n6}.

\begin{figure}%
  \centering
  \includegraphics[width = 0.20\textwidth]{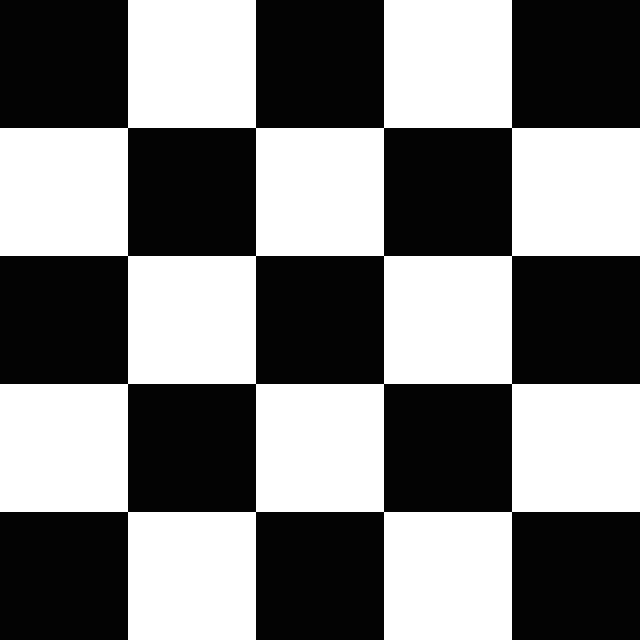}
  \caption{A simplistic test image. Each black or white square is $128 \times 128$ pixels.}
  \label{fig:pixel_true}
\end{figure}

\begin{figure*}%
  \centering
  \subfloat[Cepstral matching, $n = 4$. \label{fig:pixel_ceps_n4}]{\includegraphics[width=0.20\textwidth]{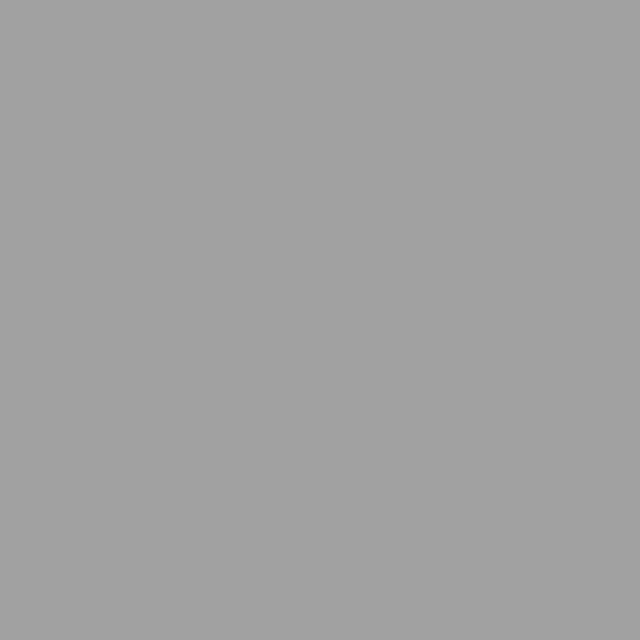}}%
  \hfil
  \subfloat[Cepstral matching, $n = 10$. \label{fig:pixel_ceps_n10}]{\includegraphics[width=0.21\textwidth]{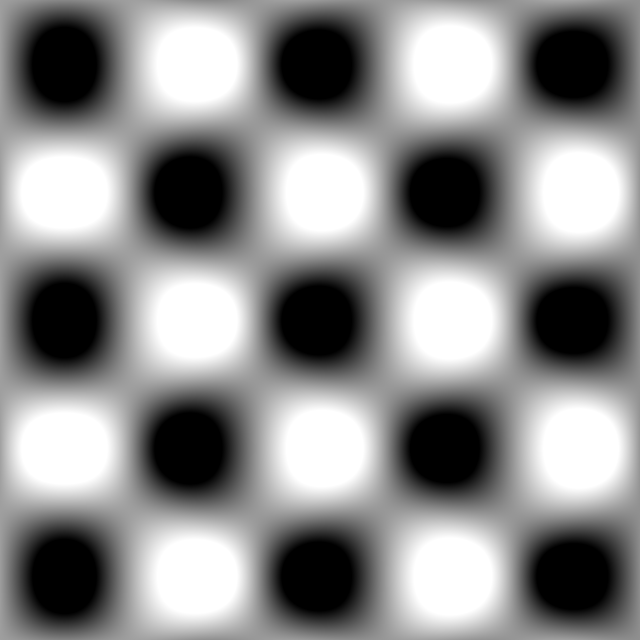}}%
  \hfil
  \subfloat[Cepstral matching, $n = 11$. \label{fig:pixel_ceps_n11}]{\includegraphics[width=0.21\textwidth]{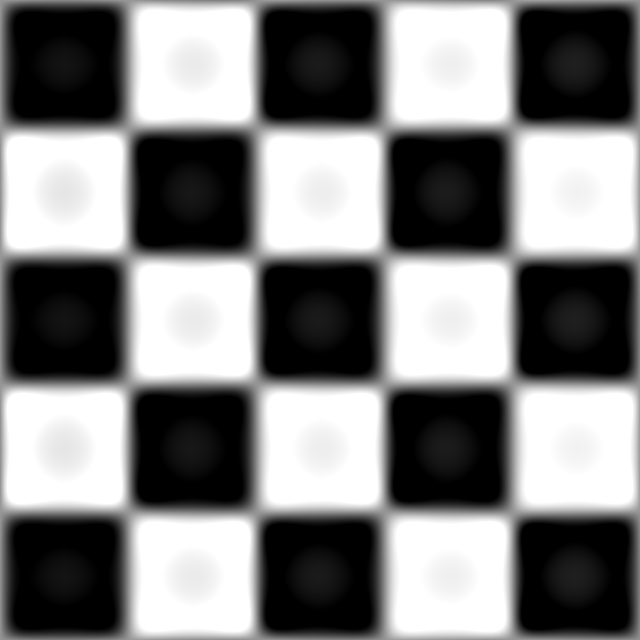}}%
  \hfil
  \subfloat[Cepstral matching, $n = 30$. \label{fig:pixel_ceps_n30}]{\includegraphics[width=0.21\textwidth]{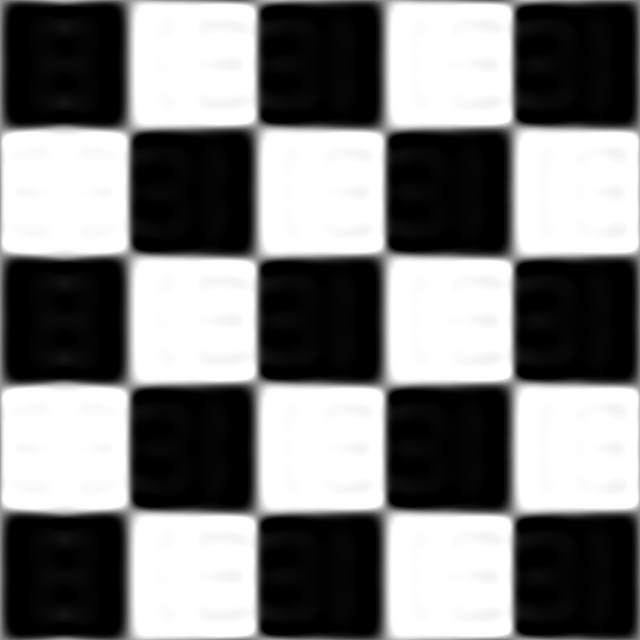}} \\
  \subfloat[ME solution, $n = 6$. \label{fig:pixel_me_n6}]{\includegraphics[width=0.21\textwidth]{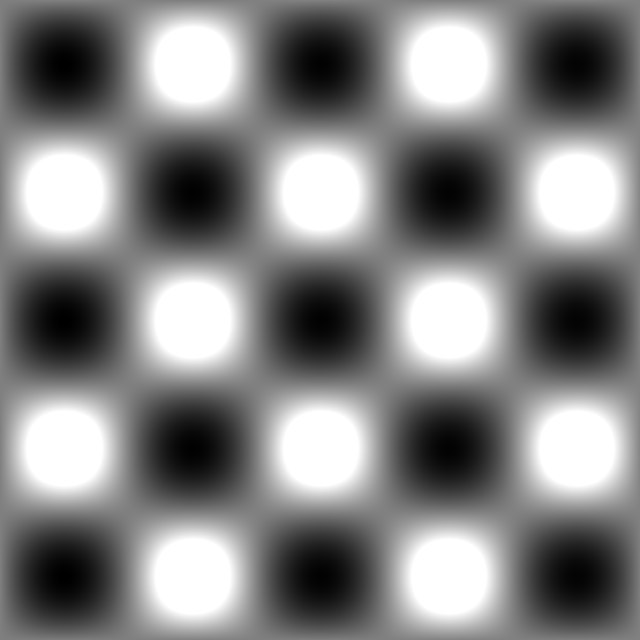}}%
  \hfil
  \subfloat[ME solution, $n = 15$. \label{fig:pixel_me_n15}]{\includegraphics[width=0.21\textwidth]{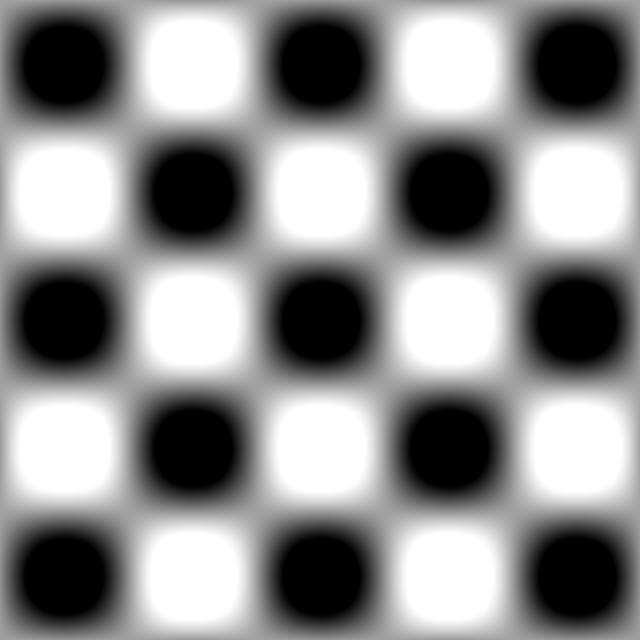}}%
  \hfil
  \subfloat[ME solution, $n = 16$. \label{fig:pixel_me_n16}]{\includegraphics[width=0.21\textwidth]{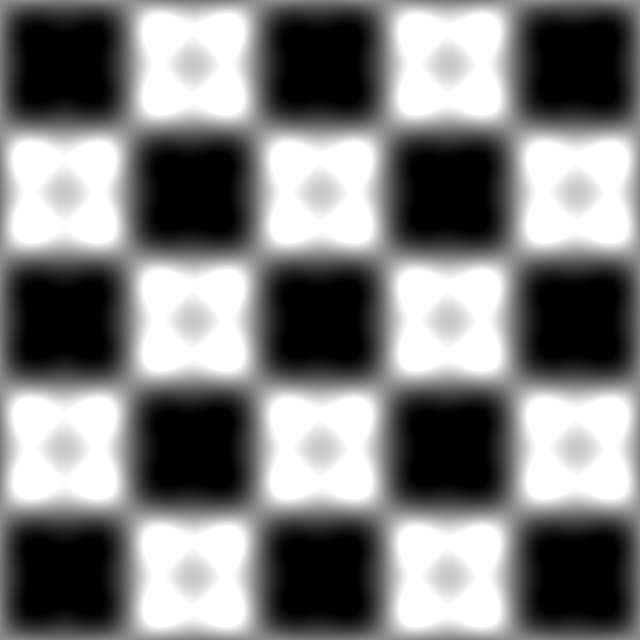}}%
  \hfil
  \subfloat[ME solution, $n = 43$. \label{fig:pixel_me_n43}]{\includegraphics[width=0.21\textwidth]{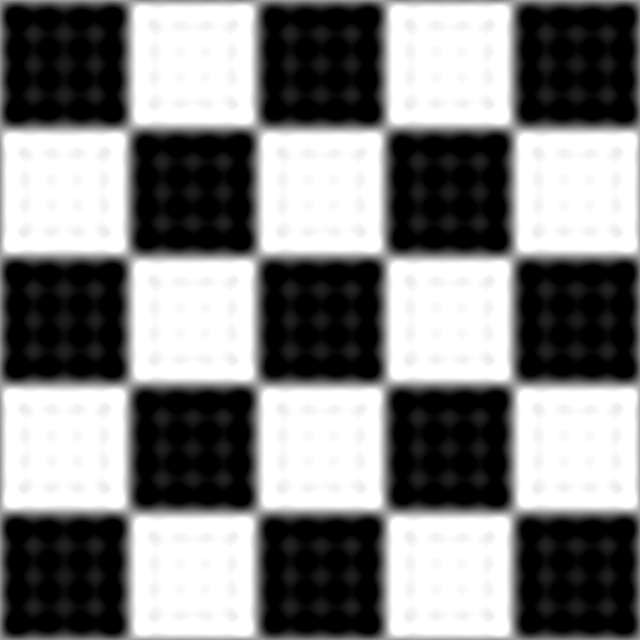}}%
  \caption{Compressions of the simple image shown in Figure~\ref{fig:pixel_true}. The top row shows compression with regularized covariance and cepstral matching, where $\lambda = 10^{-2}$, and the bottom row shows compression with the maximum-entropy solution. In all cases $n_1 = n_2$, and the pair of compressions in each column have approximately the same number of parameters, namely $n_{\rm me} \approx \sqrt{2} \, n_{\rm ceps}$.}%
  \label{fig:pixelImages}%
\end{figure*}

\begin{figure}[!t]
\centering
\includegraphics[width = 0.55 \columnwidth]{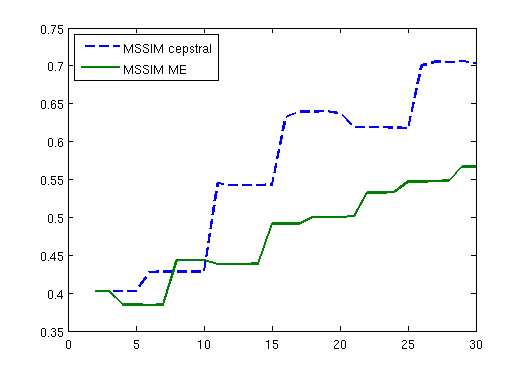}
\caption{MSSIM values of different compression levels, plotted against $n$ for the compression with cepstral matching. Hence the corresponding ME compression has $\lceil \sqrt{2} n \rceil$ coefficients.}
\label{fig:MSSIM_graph}
\end{figure}

\subsection{Compression of real images}

\begin{figure}[t]
  \centering
  \subfloat[Original image. \label{fig:firstPage}]{\includegraphics[width=0.30\columnwidth]{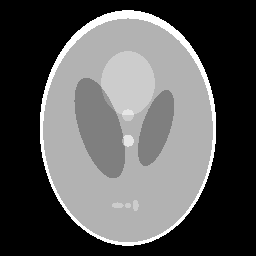}}%
  \hfil
  \subfloat[Cepstral matching, $n=30$ and $\lambda = 10^{-2}$. \label{fig:firstPage2}]{\includegraphics[width=0.30\columnwidth]{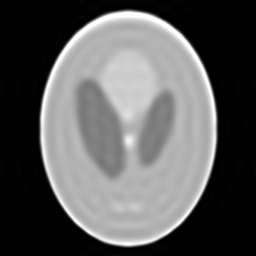}}%
  \hfil
  \subfloat[ME solution, $n = 45$. \label{fig:firstPage3}]{\includegraphics[width=0.30\columnwidth]{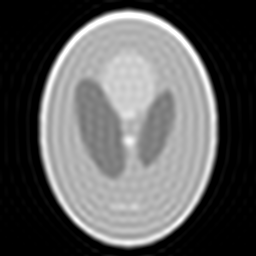}}%
  \caption{Compression of the Shepp-Logan phantom, with a compression rate of $97\%$.}%
  \label{fig:firstPage4}%
  \hfil
  \subfloat[Original image.\label{fig:true_lenna}]{\includegraphics[width=0.30\textwidth]{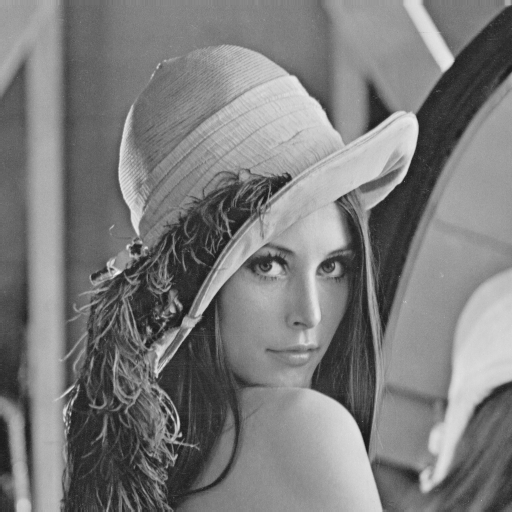}}%
  \hfil
  \subfloat[Cepstral matching, $n=60$ and $\lambda = 10^{-2}$. \label{fig:cep_lenna}]{\includegraphics[width=0.30\textwidth]{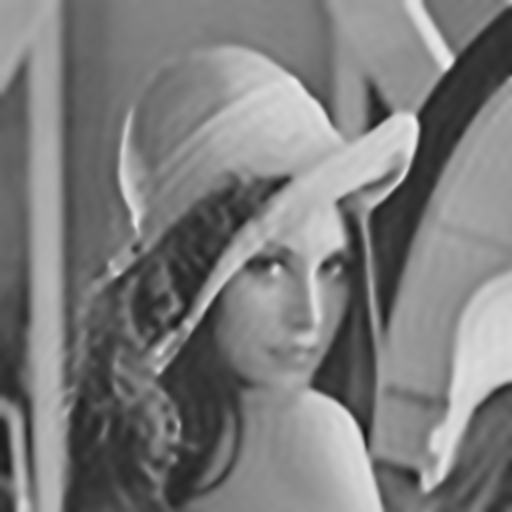}}%
  \hfil
  \subfloat[ME solution, $n = 85$. \label{fig:ME_lenna}]{\includegraphics[width=0.30\textwidth]{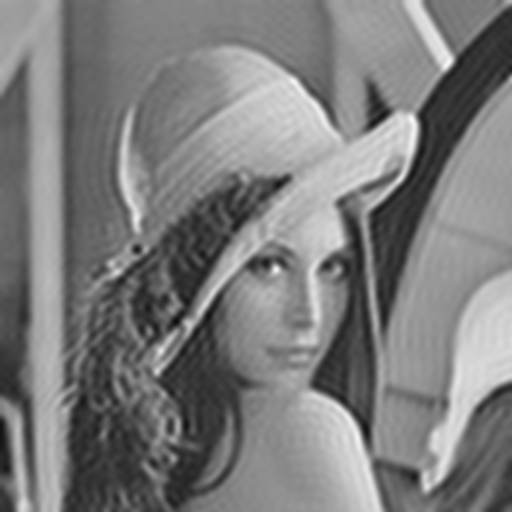}}%
  \caption{Compression of the Lenna image, with a compression rate of about $97\%$.}%
  \label{fig:Shepp_Lenna}%
\end{figure}

We now apply the methods to some more realistic images. In the first example, shown in Figure~\ref{fig:firstPage}
the original image 
is the Shepp-Logan phantom often used in medical imaging \cite{shepp1974fourier}, of size $256 \times 256$ pixels. In 
Figure~\ref{fig:firstPage2}
a compression using covariance and cepstral mathing is shown, where $n_1 + 1 = n_2 + 1 = 30$.
Hence this image is described by $2 \cdot 30^2 = 1800$ parameters, compared to the original $256^2 = 65536$ parameters, which corresponds to a reduction in parameters of about $97\%$.
We  also compute an 
ME compression, with degree $n_1 + 1 = n_2 + 1 = 45 \approx \sqrt{2} \cdot 30$ which is shown in Figure~\ref{fig:firstPage3}.

The second example is a compression of the classical Lenna image, often used in the image processing literature. The original image, shown in Figure~\ref{fig:true_lenna}, is $512 \times 512$ pixels. For regularized cepstral matching we set $n_1 + 1 = n_2 + 1 = 60$,  corresponding to a compression rate of about $97\%$, and the result is shown in Figure~\ref{fig:cep_lenna}. The ME compression, computed with $n_1 + 1 = n_2 + 1 = 85 \approx \sqrt{2} \cdot 60$,  is shown in Figure~\ref{fig:ME_lenna}.

The MSSIM values for these compressions are shown in Table \ref{tab:mssim}. 
They seem to agree with the visual impression. Interestingly the compression with cepstral matching is better for the Shepp-Logan phantom. 
However, in  the Lenna image neither of the methods outperform the other. The ME compression has more ringing artifacts, but it is less blurred than the cepstral compression. 
We believe that this is related to the fact that if you have relatively few sharp transitions in pixel values, which is the case in Figure~\ref{fig:pixel_true} and Figure~\ref{fig:firstPage},
placing both poles and zero close to each other can achieve this transition efficiently and thus give better quality on the compressed image. However when this is not the case, as with the Lenna image, the trade-off between having spectral zeros or matching higher frequencies is more complex.
\begin{table}[h]
\footnotesize
\caption{MSSIM-values of different compression techniques, on the two test images.}
\label{tab:mssim}
\vspace{-10pt}
\begin{center}
\begin{tabular}{l l l l}
\toprule
\multicolumn{2}{c}{Shepp-Logan} & \multicolumn{2}{c}{Lenna} \\
\cmidrule(r){1-2}
\cmidrule(r){3-4}
Compression  & MSSIM-value  & Compression  & MSSIM-value \\
\midrule
Cepstral     & 0.8690       & Cepstral     & 0.7451 \\
ME           & 0.7044       & ME           & 0.7489 \\
\bottomrule
\end{tabular}
\end{center}
\vspace{-10pt}
\end{table} 

Similar methods have previously been used for compression of textures \cite{Chiuso-F-P-05,picci2008modelling}, where, instead of a scalar two-dimensional moment problem, a one-dimensional vector problem is considered. Here the image is modeled by a periodic stochastic vector process rather than a two-dimensional random field, leading to a discrete vector moment problem akin to the one presented in \cite{lindquist2013thecirculant}. This is connected to the circulant moment problem considered in Section~\ref{subsec:periodProb} and to modeling of reciprocal systems \cite{levy1990modeling,Carli-F-P-P-11}.  

\appendix\section*{}

In this appendix we provide the proofs that have been deferred in the main text.
Some of the proofs use general properties of multidimensional trigonometric polynomials, summarized in this lemma.
\begin{lemma}\label{lem:trigPoly}
For all  $P \in \bar{\mathfrak{P}}_+$ we have i) $|p_{k_1 \ldots ,k_d}| \leq p_{0 \ldots, 0}$ and ii) $\|P\|_\infty \leq |\Lambda| \|p\|_\infty$.
\end{lemma}

\begin{proof}
The fact that
$|p_\kb| = \left| \int_{\mathbb{T}^ d} e^{i (\kb,\thetab)} P \dm \right| \leq \int_{\mathbb{T}^ d} |e^{i (\kb,\thetab)}| \, |P| \dm  = p_{0}$ implies i).
%
%
Next we note that $P$ has $|\Lambda|$ coefficients, and hence
\begin{displaymath}
\|P\|_\infty\leq  \sup_{\thetab \in \mathbb{T}^d} \sum_{\kb \in \Lambda} |p_{k}|  |e^{i (\kb,\thetab)}| = \sum_{\kb \in \Lambda} |p_\kb|  \leq  \displaystyle |\Lambda| \|p\|_\infty,
\end{displaymath}
which proves ii).
\end{proof}

\begin{proofWithName}{Proof of Lemma \ref{lem:Jcont}}
To show lower semicontinuity of
\begin{equation*}
\mathbb{J}_P(Q) = \langle c, q \rangle + \int_{\mathbb{T}^d} - P \log Q\,\dm
\end{equation*}
we note that $\langle c, q \rangle$ is continuous and hence only the integral needs to be considered. 

Fix any $Q \in \bar{\mathfrak{P}}_+ \setminus \{0\}$. From \cite[p. 223]{schinzel2000polynomials} we know that it is log-integrable. Moreover, let $(Q_n)$ be a sequence of trigonometric polynomials in $\bar{\mathfrak{P}}_+ \setminus \{0\}$ 
that converges to $Q$ in $\LInf$.
We know that $Q$ is bounded, and, since the convergence $Q_n\to Q$ is uniform, we must  have $M:=\sup_n \{ \max_\thetab [ Q_n ] \}< \infty$, 
and thus $0 \leq Q/M \leq 1$ and $0 \leq Q_n/M \leq 1$ for all $n$. Moreover, $\lim_{n\rightarrow \infty} -\log (Q_n/M)=-\log (Q/M)$
in extended real-valued sense.
Since  $-\log ( Q_n/M ) \geq 0$,  by Fatou's Lemma \cite[p. 23]{rudin1987real}, we have 
\[
\int_{\mathbb{T}^d} -\log \left(\frac{Q}{M}\right) \dm \leq \liminf_{n \rightarrow \infty} \int_{\mathbb{T}^d} -\log \left(\frac{Q_n}{M}\right) \dm.
\]
Since $(Q_n)$ is an arbitrary sequence, the functional is lower semicontinuous in $Q$. Moreover, since $Q$ is also arbitrary 
it follows that $\mathbb{J}_P$ is lower semicontinuous on $\bar{\mathfrak{P}}_+ \setminus \{0\}$.
\end{proofWithName}

\begin{proofWithName}{Proof of Proposition \ref{prop:d>2}}  
Let $\kb_1,\kb_2,\kb_3\in \Lambda$  be three linearly independent index
vectors. First note that the trigonometric polynomial
$Q(e^{i\thetab})=\sum_{\ell=1}^3(1-(e^{i(\kb_\ell,\thetab)}+e^{-i(\kb_\ell,
\thetab)})/2)$ is nonnegative and  $Q(e^{i{\bf 0}})=0$, hence $Q\in
\partial \mathfrak{P}_+$. Next we will show that $\int_{\mathbb{T}^d}
Q^{-1}\dm(\thetab)$ is finite. By the variable change ${\boldsymbol
\phi}=A \thetab$, where $A\in \mR^{d\times d}$ is selected to be
invertible and with $\ell$th row equal to $\kb_\ell$ for $\ell=1,2,3$, the
integral becomes 
\begin{equation*}
\int_{\mathbb{T}^d} \frac{1}{Q}\dm(\thetab)=\int_{A(\mathbb{T}^d)}
\frac{\det(A)^{-1}}{\sum_{\ell=1}^3(1-\cos(\phi_\ell))}\dm({\boldsymbol
\phi}),
\end{equation*}
where the set $A(\mathbb{T}^d)=\{A\thetab \mid \thetab\in \mT^d\}$.
Due to the periodicity of the integrand, the integral is bounded by $$\kappa\int_{\mathbb{T}^3}
\frac{d\phi_1d\phi_2d\phi_3}{\sum_{\ell=1}^3(1-\cos(\phi_\ell))}$$ for
some constant $\kappa$ that depends on $A$ and $d$. This
bound is finite  \cite{lang1982multidimensional,
karlsson2015themultidimensional}, and therefore the proposition follows.
\end{proofWithName}

To prove Theorem \ref{theo:fp}, we need the following lemma. 

\begin{lemma}\label{lem:fBiject}
$f^p$ is a bijective map.
\end{lemma}

\begin{proof}
By Corollary 
\ref{theo:langMcclellan}, $f^p$ is injective, since there is a unique minimizer of \eqref{dualfunctional}  over all $Q\in\mathfrak{P}_+$. Hence there is at most one $q$ corresponding to a certain $c$, proving  injectivity. 
Surjectivity also follows from Corollary 
\ref{theo:langMcclellan}. We fix a $P \in \mathfrak{P}_+$ and simply note that there exist a unique solution for all $c \in \mathfrak{C}_+$, given by $q = (f^p)^{-1}(c)$.
\end{proof}

\begin{proofWithName}{Proof of Theorem \ref{theo:fp}}
In the proof of Theorem~\ref{theo:conjecture} we saw that $\partial^2 \mathbb{J}_P(Q ; \delta Q) > 0$ for all nontrivial variations $\delta Q$. Hence 
\begin{equation}\label{eq:firstDerivFP}
\frac{\partial f^p_\kb}{\partial q_\lbb} = \int_{\mathbb{T}^d} e^{i(\kb-\lbb, \thetab)} \frac{P}{Q^2} \dm =\frac{\partial^2 \mathbb{J}_P(Q)}{\partial q_\lbb \partial \bar{q}_\kb}
\end{equation}
is positive definite. Next, we define the map $\varphi^p: \mathfrak{C}_+ \times \mathfrak{P}_+ \rightarrow   
\{(r_\kb)_{\kb\in \Lambda}\in \mC^{|\Lambda|}\,|\, r_{-\kb}=\bar r_{\kb}, \kb\in \Lambda\}\cong 
\mathbb{R}^{|\Lambda|}$ as
\begin{equation*}
\varphi^p_\kb (c,q) = c_\kb - \int_{\mathbb{T}^d} e^{i (\kb,\thetab)} \frac{P}{Q} \dm.
\end{equation*}
By Corollary~\ref{theo:langMcclellan}, $\gamma(c,q) = 0$ has a unique solution for each $c \in \mathfrak{C}_+$.
Since $\partial\varphi^p /\partial q = \partial f^p/\partial q$ is invertible, the Implicit Function Theorem implies that  $q = (f^p)^{-1}(c)$ is locally a $\mathcal{C}^1$ function and hence a local diffeomorphism. However, $f^p$ is a bijection (Lemma~\ref{lem:fBiject}) and therefore   a (global) diffeomorphism.
\end{proofWithName}

By Theorem~\ref{theo:fp}, the function $g^c$ is a well-defined map. The proof of Theorem \ref{theo:gc} now follows along the same lines.

\begin{lemma}\label{lem:gBij}
$g^c$ is a bijective map.
\end{lemma}

\begin{proof}
Surjectivity of $g^c$  on the image $\mathfrak{Q}_+$ follows directly from  definition. A straight-forward generalization of Lemma~2.4 in \cite{byrnes2006thegeneralized} shows that $g^c$ is injective.
\end{proof}

\begin{proofWithName}{Proof of Theorem \ref{theo:gc}}
Let the map $\varphi^c: \mathfrak{P}_+ \times \mathfrak{P}_+ \rightarrow \{(r_\kb)_{\kb\in \Lambda}\in \mC^{|\Lambda|}\,|\, r_{-\kb}=\bar r_{\kb}, \kb\in \Lambda\}\cong 
\mathbb{R}^{|\Lambda|}$ be given by
\begin{equation*}
\displaystyle \varphi^c_\kb (p,q) = c_\kb - \int_{\mathbb{T}^d} e^{i(\kb, \thetab)} \frac{P}{Q} \dm.
\end{equation*}
The Jacobian with respect to $q$ is the same as \eqref{eq:firstDerivFP}.
Hence $q = g^c(p)$ is $\mathcal{C}^1$ by the  Implicit Function Theorem.
Since \eqref{eq:firstDerivFP} gives a positive definite Jacobian matrix,
\begin{equation*}
\frac{\partial\varphi^c_\kb}{\partial p_\lbb} = - \int_{\mathbb{T}^d} e^{i(\kb-\lbb,\thetab)} \frac{1}{Q} \dm
\end{equation*}
defines a  invertible Jacobian. Hence $p = (g^c)^{-1}(q)$ is $\mathcal{C}^1$, so $g^c$ is a local diffeomorphism. Since it is a bijection (Lemma~\ref{lem:gBij}), it is  a (global) diffeomorphism.
\end{proofWithName}

\begin{proofWithName}{Proof of Lemma \ref{lem:JcontCeps}}
For any $Q \in \bar{\mathfrak{P}}_+ \setminus \{0\}$, $\log Q$ is integrable \cite[p. 223]{schinzel2000polynomials}.
Since $P \in \bar{\mathfrak{P}}_{+,\circ}$, $P$ is not the zero-polynomial, hence, since  $x\log x \rightarrow 0$ as $x \rightarrow 0$, $P \log P$ is integrable and in fact continuous for all $P \in \bar{\mathfrak{P}}_{+,\circ}$. Hence
\[
\int_{\mathbb{T}^d}\!\! P \log P\, \dm - \int_{\mathbb{T}^d} P \log Q\, \dm = \int_{\mathbb{T}^d}\!\! P \log\left(\frac{P}{Q}\right) \dm,
\]
and therefore we can rewrite the functional $\mathbb{J}(P,Q)$ as
\[
\mathbb{J}(P,Q) = \langle c, q \rangle - \langle \gamma, p \rangle + \int_{\mathbb{T}^d}\!\!  P \log P\, \dm - \int_{\mathbb{T}^d}  \!\!  P\log Q\, \dm.
\]
All terms in this expression are continuous, except possibly the last integral.
However, following along the same lines as in the proof of Lemma \ref{lem:Jcont}, 
 we can apply Fatou's Lemma showing that $\mathbb{J}(P,Q)$ is lower semicontinuous.
\end{proofWithName}

\begin{proofWithName}{Proof of Lemma \ref{lm:jCompSublevel}}
To show that $\mathbb{J}^{-1}(-\infty, r]$ have compact sublevel sets, we proceed as in \cite[p. 503]{LindquistPicci2015} by first splitting the objective function into two parts
\begin{displaymath}
\mathbb{J}_1(P,Q) = \langle c, q \rangle - \int_{\mathbb{T}^d} P \log  Q \, \dm \quad\text{and}\quad
\mathbb{J}_2(P) = - \langle \gamma, p \rangle + \int_{\mathbb{T}^d} P \log  P\, \dm .
\end{displaymath}
The sublevel set consists of the $(P, Q) \in \bar{\mathfrak{P}}_{+,\circ} \times \bar{\mathfrak{P}}_+$ such that
$r \geq \mathbb{J}_1(P,Q) + \mathbb{J}_2(P)$,
and from Lemma \ref{lemma:linearAndLogarithmicGrowth} we have
$\mathbb{J}_1(P,Q) \geq \varepsilon \|Q\|_\infty + \ \log\|Q\|_\infty$,
since $\int_{\mathbb{T}^d} P \dm = 1$ by \eqref{expintegral}.
Next we show that  $\mathbb{J}_2(P)$ is bounded from below. We first note that since $P \in \bar{\mathfrak{P}}_{+,\circ}$ we have  $p_{0} = 1$, and thus $P$ is bounded away from the zero polynomial. Now, since $x \log(x)$ achieves a minimum $> -\infty$ on any compact set $[0, a]$, $P \log P$ must achieve a minimum $> -\infty$ on $\mathbb{T}^d$. Calling this minimum $\kappa_P$, we have
\begin{equation*}
\int_{\mathbb{T}^d} P \log  P\, \dm \geq \int_{\mathbb{T}^d} \kappa_P \dm = \kappa_P
\end{equation*}
To bound the term $- \langle \gamma, p \rangle$ from below we note that
\begin{equation*} 
\langle \gamma, p \rangle = \sum_{\kb \in \Lambda} \bar{\gamma}_\kb p_\kb \leq \left| \sum_{\kb \in \Lambda} \bar{\gamma}_\kb p_\kb \right| \leq \sum_{\kb \in \Lambda} |\bar{\gamma}_\kb| \,  |p_\kb| 
\leq \sum_{\kb \in \Lambda} \|\gamma\|_\infty  |p_\kb| \leq \|\gamma\|_\infty |\Lambda| \|p\|_\infty
\end{equation*}
and thus $- \langle \gamma, p \rangle \geq -|\Lambda| \|\gamma\|_\infty \|p\|_\infty=-|\Lambda| \|\gamma\|_\infty$, since $\|p_\infty\| = p_0 = 1$ by Lemma \ref{lem:trigPoly}. Hence there exist some $\rho > -\infty$ such that $\mathbb{J}_2(P) \geq \rho$. From this we have
\begin{equation*}
r - \rho \geq \mathbb{J}_1(P,Q) \geq \varepsilon \|Q\|_\infty + \log\|Q\|_\infty ,
\end{equation*}
so comparing linear and logarithmic growth we see that the set is bounded both from above and below. As before, since it is the sublevel set of a lower semicontinuous function it will be closed, and hence it is compact.
\end{proofWithName}

\begin{proofWithName}{Proof of Lemma \ref{lm:jCepsConv}}
Consider the directional derivative of $\mathbb{J}$ in a point $(P,Q) \in \bar{\mathfrak{P}}_{+,\circ} \times \in \bar{\mathfrak{P}}_{+}$ in any direction $(\delta P, \delta Q)$ such that $P + \varepsilon \delta P \in \bar{\mathfrak{P}}_{+,\circ}$, and $Q + \varepsilon \delta Q \in \bar{\mathfrak{P}}_{+}$ for all  $\varepsilon\in (0, a)$  for some $a > 0$. A quite straight-forward calculation yields 
\begin{equation*}
\delta \mathbb{J}(P,Q; \delta P, \delta Q) =\langle c, \delta q \rangle - \langle \gamma, \delta p \rangle
+ \int_{\mathbb{T}^d} \left[\delta P \log \left( \frac{P}{Q} \right) - \delta Q \frac{P}{Q}\right] \dm.
\end{equation*}
where we have used the fact, obtained from \eqref{expintegral}, that 
$\int_{\mathbb{T}^d} \delta P \dm = \delta p_{0} = 0$,
since $p_0=1$ is constant. Likewise, the  second directional derivative becomes
\begin{equation*}
\begin{split}
 \delta^2 \mathbb{J}(P,Q; \delta P, \delta Q) =\int_{\mathbb{T}^d} P \left( \delta P \frac{1}{P} - \delta Q \frac{1}{Q} \right)^2 \dm,
\end{split}
\end{equation*}
which is clearly nonnegative for all feasible directions and hence positive semi-definite. Thus the problem is convex.
\end{proofWithName}

\begin{proofWithName}{Proof of Lemma \ref{lm:limN}} 
First note that $\CP(\Nb)\subset\CP$. To prove the lemma, it is sufficient to prove that any $c\in \CP$ belongs to $\CP(\Nb)$ if $\min(\Nb)$ is large enough. 

Let $c\in \CP$. From \eqref{eq:mcineq} there exists $\kappa_c>0$ such that 
\begin{equation}\label{eq:mc2}
\langle c,p \rangle \ge  \kappa_c \|p\|_\infty, \quad \mbox{ for all } p\in \PPC.
\end{equation}
We want to show that $\langle c,\hat p\rangle>0$ for any $\hat p\in \PPC(\Nb)\setminus\{0\}$. Without loss of generality we may take 
$\|\hat p\|_\infty=1$. Then  $|\partial\hat P(e^{i\thetab})/\partial\theta_j|\le \sum_{\kb\in \Lambda}|k_j|$, and, since $\hat P(e^{i\thetab})\ge 0$ in $\thetab\in \mT_\Nb$,  it follows that $\hat P(e^{i\thetab})\ge -\pi\Delta/\min(\Nb)$ where $\Delta=\sum_{\kb\in \Lambda}\|\kb\|_1$. Therefore $\hat P+\pi\Delta/\min(\Nb)\in \PPC$, and by using \eqref{eq:mc2} we get 
\begin{equation*}
\langle c,\hat p \rangle +c_0 \frac{\pi\Delta}{\min(\Nb)} \ge  \kappa_c \left(\|\hat p\|_\infty-\frac{\pi\Delta}{\min(\Nb)}\right).
\end{equation*}
By selecting $\min(\Nb)>\pi\Delta(1+c_0/\kappa_c)$, we obtain $\langle c,\hat p \rangle >0$. 
Since $\hat p\in\PPC(\Nb)\setminus\{0\}$  is arbitrary, it therefore follows that $c\in \CP(\Nb)$.
\end{proofWithName}

\begin{proofWithName}{Proof of Lemma \ref{lem:seqBounded}}
For a fixed $\tilde{Q} \in \mathfrak{P}_+$ we have 
$\lim_{\min(\Nb) \rightarrow \infty} \mathbb{J}_P^{\Nb} (\tilde{Q}) = \mathbb{J}_P (\tilde{Q})$,
since the sums in \eqref{eq:JdiscConv} are Riemann sums converging to \eqref{eq:JcontConv}. Hence we can define
$L := \sup_{N} \mathbb{J}_P^{\Nb} (\tilde{Q}) < \infty$.
Also, by optimality, $\infty > \mathbb{J}_P^{\Nb} (\tilde{Q}) \geq \mathbb{J}_P^{\Nb} (\hat{Q}_{\Nb})$ for all values of $\Nb$
and also $\infty > \mathbb{J}_P (\tilde{Q}) \geq \mathbb{J}_P (\hat{Q})$. 
Using this and Lemma~\ref{lemma:linearAndLogarithmicGrowth} we obtain
\begin{displaymath}
L \geq\mathbb{J}_P^{\Nb} (\tilde{Q}) \geq \mathbb{J}_P^{\Nb} (\hat{Q}_{N})\geq \varepsilon_{\Nb} \| \hat{Q}_{\Nb}\|_\infty - \|P\|_1 \| \log(\hat{Q}_{\Nb})\|_\infty
\end{displaymath}
for all values of $\Nb$.
In accordance with \eqref{eq:cqApprox2},  we can choose $\varepsilon_{\Nb} := \kappa_c^{\Nb}/|\Lambda|$, where $\kappa_c^{\Nb}$ is the minimum value of $\langle c,q_\Nb \rangle$ on the compact set $\{Q \in \bar{\mathfrak{P}}_+(\Nb) \mid \|q\|_{\infty} = 1\}$.
 If we can show $\kappa_c := \inf_{\Nb} \kappa_c^{\Nb} > 0$, we can choose $\varepsilon := \kappa_c/|\Lambda| \leq \varepsilon_\Nb$ for all $\Nb$, so that
\begin{equation*}
\begin{array}{l}
L \geq \varepsilon \|\hat{Q}_\Nb\|_\infty - \|P\|_1 \| \log(\hat{Q}_{\Nb})\|_\infty .
\end{array}
\end{equation*}
Then comparing linear and logarithmic growth this implies that  $(\hat Q_{\Nb})$ is bounded.

To show that $\kappa_c > 0$ first note that for every finite value of $\min(\Nb)$ we have $\kappa_c^{\Nb} > 0$. Now assume $\inf_{\Nb} \kappa_c^{\Nb} = 0$. Then there must exist a sequence $(q_\Nb^\star)$ such that $\langle c, q_{\Nb}^\star \rangle \rightarrow 0$ as $\min(\Nb) \rightarrow \infty$, where
$q_{\Nb}^\star\in \bar{\mathfrak{P}}_+(\Nb)$ and $\|q_\Nb^\star\|_\infty = 1$.
Now, since every $q_{\Nb}^\star$ is  a vector in $\mathbb{C}^{|\Lambda|}$, the constraint $\|q\|_\infty = 1$ defines a compact set. Hence there is a subsequence, also indexed with $\Nb$, so that $q^\star := \lim_{\min(\Nb)\to\infty} q_{\Nb}^\star$ is well-defined and $\|q^\star\|_\infty = 1$. Then $\langle c, q^\star \rangle = 0$. However, since $c \in \mathfrak{C}_+$ and $q^\star\in \bar{\mathfrak{P}}_+$, this implies that $q^\star = 0$, which contradicts $\|q^\star\|_\infty = 1$. Hence $\kappa_c > 0$, as claimed. 
\end{proofWithName}

\bibliographystyle{siam}
\bibliography{ref}

\begin{thebibliography}{10}

\bibitem{avventi2011spectral}
{\sc E.~Avventi}, {\em Spectral Moment Problems : Generalizations,
  Implementation and Tuning}, PhD thesis, 2011.
\newblock Optimization and {S}ystems {T}heory, {D}epartment of {M}athematics,
  {KTH} {R}oyal {I}nstitue of {T}echnology.

\bibitem{blomqvist2003matrix}
{\sc A.~Blomqvist, A.~Lindquist, and R.~Nagamune}, {\em Matrix-valued
  {N}evanlinna-{P}ick interpolation with complexity constraint: an optimization
  approach}, IEEE Transactions on Automatic Control, 48 (2003), pp.~2172--2190.

\bibitem{bose2003multidimensional}
{\sc N.K. Bose}, {\em Multidimensional Systems Theory and Applications}, Kluwer
  Academic Publishers, second~ed., 2003.

\bibitem{burg1967maximum}
{\sc J.P. Burg}, {\em Maximum entropy spectral analysis}, in Proceedings of the
  37th Meeting Society of Exploration Geophysicists, 1967.

\bibitem{burg1975maximum}
\leavevmode\vrule height 2pt depth -1.6pt width 23pt, {\em Maximum Entropy
  Spectral Analysis}, PhD thesis, 1975.
\newblock Department of Geophysics, Stanford University.

\bibitem{byrnes2001cepstral}
{\sc C.I. Byrnes, P.~Enqvist, and A.~Lindquist}, {\em Cepstral coefficients,
  covariance lags, and pole-zero models for finite data strings}, IEEE
  Transactions on Signal Processing, 49 (2001), pp.~677--693.

\bibitem{byrnes2002identifyability}
\leavevmode\vrule height 2pt depth -1.6pt width 23pt, {\em Identifiability and
  well-posedness of shaping-filter parameterizations: A global analysis
  approach}, SIAM Journal on Control and Optimization, 41 (2002), pp.~23--59.

\bibitem{byrnes2000anewapproach}
{\sc C.I. Byrnes, T.T. Georgiou, and A.~Lindquist}, {\em A new approach to
  spectral estimation: a tunable high-resolution spectral estimator}, IEEE
  Transactions on Signal Processing, 48 (2000), pp.~3189--3205.

\bibitem{byrnes2001ageneralized}
\leavevmode\vrule height 2pt depth -1.6pt width 23pt, {\em A generalized
  entropy criterion for {Nevanlinna-Pick} interpolation with degree
  constraint}, IEEE Transactions on Automatic Control, 46 (2001), pp.~822--839.

\bibitem{byrnes2006generalizedinterpolation}
{\sc C.I. Byrnes, T.T. Georgiou, A.~Lindquist, and A.~Megretski}, {\em
  Generalized interpolation in ${H}^\infty$ with a complexity constraint},
  Transactions of the American Mathematical Society, 358 (2006), pp.~965--987.

\bibitem{byrnes1998aconvex}
{\sc C.I. Byrnes, S.V. Gusev, and A.~Lindquist}, {\em A convex optimization
  approach to the rational covariance extension problem}, SIAM Journal on
  Control and Optimization, 37 (1998), pp.~211--229.

\bibitem{byrnes2001fromfinite}
\leavevmode\vrule height 2pt depth -1.6pt width 23pt, {\em From finite
  covariance windows to modeling filters: A convex optimization approach}, SIAM
  Review, 43 (2001), pp.~645--675.

\bibitem{byrnes2003aconvex}
{\sc C.I. Byrnes and A.~Lindquist}, {\em A convex optimization approach to
  generalized moment problems}, in Control and Modeling of Complex Systems,
  Koichi Hashimoto, Yasuaki Oishi, and Yutaka Yamamoto, eds., Trends in
  Mathematics, Birkhäuser, Boston, 2003, pp.~3--21.

\bibitem{byrnes2006thegeneralized}
\leavevmode\vrule height 2pt depth -1.6pt width 23pt, {\em The generalized
  moment problem with complexity constraint}, Integral Equations and Operator
  Theory, 56 (2006), pp.~163--180.

\bibitem{Byrnes-L-09}
\leavevmode\vrule height 2pt depth -1.6pt width 23pt, {\em The moment problem
  for rational measures: convexity in the spirit of {K}rein}, in Modern
  Analysis and Application: Mark Krein Centenary Conference, Vol. I: Operator
  Theory and Related Topics, vol.~190 of Operator Theory Advances and
  Applications, Birkh{\"a}user, 2009, pp.~157--169.

\bibitem{byrnes1995acomplete}
{\sc C.I. Byrnes, A.~Lindquist, S.V. Gusev, and A.S. Matveev}, {\em A complete
  parameterization of all positive rational extensions of a covariance
  sequence}, IEEE Transactions on Automatic Control, 40 (1995), pp.~1841--1857.

\bibitem{Carli-F-P-P-11}
{\sc F.~P. Carli, A.~Ferrante, M.~Pavon, and G.~Picci}, {\em A maximum entropy
  solution of the covariance extension problem for reciprocal processes},
  Automatic Control, IEEE Transactions on, 56 (2011), pp.~1999--2012.

\bibitem{Chiuso-F-P-05}
{\sc A.~Chiuso, A.~Ferrante, and G.~Picci}, {\em Reciprocal realization and
  modeling of textured images}, in 44th IEEE Conference on Decision and Control
  (CDC), and European Control Conference (ECC), Dec 2005, pp.~6059--6064.

\bibitem{deller2000discrete}
{\sc J.R. Deller, J.G. Proakis, and J.H.L. Hansen}, {\em Discrete-time
  processing of speech signals}, IEEE Press, Piscataway, N.Y., 2000.

\bibitem{dickinson1980two-dimensional}
{\sc B.~Dickinson}, {\em Two-dimensional markov spectrum estimates need not
  exist}, IEEE Transactions on Information Theory, 26 (1980), pp.~120--121.

\bibitem{dumitrescu2007positive}
{\sc B.~Dumitrescu}, {\em Positive Trigonometric Polynomials and Signal
  Processing Applications}, Springer, Berlin, 2007.

\bibitem{ekstrom1984digital}
{\sc M.P. Ekstrom}, {\em Digital image processing techniques}, Academic Press,
  1984.

\bibitem{woods1976two-dimensionalspec}
{\sc M.P. Ekstrom and J.W. Woods}, {\em Two-dimensional spectral factorization
  with applications in recursive digital filtering}, IEEE Transactions on
  Acoustics, Speech and Signal Processing, 24 (1976), pp.~115--128.

\bibitem{enqvist2004aconvex}
{\sc P.~Enqvist}, {\em A convex optimization approach to {ARMA}(n,m) model
  design from covariance and cepstral data}, SIAM Journal on Control and
  Optimization, 43 (2004), pp.~1011--1036.

\bibitem{enqvist2007approximative}
{\sc P.~Enqvist and E.~Avventi}, {\em Approximative covariance interpolation
  with a quadratic penalty}, in Decision and Control, 2007 46th IEEE Conference
  on, 2007, pp.~4275--4280.

\bibitem{fanizza2008modeling}
{\sc G.~Fanizza}, {\em Modeling and Model Reduction by Analytic Interpolation
  and Optimization}, PhD thesis, 2008.
\newblock Optimization and {S}ystems {T}heory, {D}epartment of {M}athematics,
  {KTH} {R}oyal {I}nstitue of {T}echnology.

\bibitem{Ferrante2007further}
{\sc A.~Ferrante, M.~Pavon, and F.~Ramponi}, {\em Further results on the
  {B}yrnes-{G}eorgiou-{L}indquist generalized moment problem}, in Modeling,
  Estimation and Control, A.~Chiuso, S.~Pinzoni, and A.~Ferrante, eds.,
  Springer, 2007, pp.~73--83.

\bibitem{ferrante2008hellinger}
\leavevmode\vrule height 2pt depth -1.6pt width 23pt, {\em Hellinger versus
  {K}ullback-{L}eibler multivariable spectrum approximation}, IEEE Transactions
  on Automatic Control, 53 (2008), pp.~954--967.

\bibitem{georgiou1983partial}
{\sc T.T. Georgiou}, {\em Partial Realization of Covariance Sequences}, PhD
  thesis, 1983.
\newblock Center for {M}athematical {S}ystems {T}heory, {U}niveristy of
  {F}lorida.

\bibitem{georgiou1987realization}
\leavevmode\vrule height 2pt depth -1.6pt width 23pt, {\em Realization of power
  spectra from partial covariance sequences}, IEEE Transactions on Acoustics,
  Speech and Signal Processing, 35 (1987), pp.~438--449.

\bibitem{georgiou1999theinterpolation}
\leavevmode\vrule height 2pt depth -1.6pt width 23pt, {\em The interpolation
  problem with a degree constraint}, IEEE Transactions on Automatic Control, 44
  (1999), pp.~631--635.

\bibitem{georgiou2005solution}
\leavevmode\vrule height 2pt depth -1.6pt width 23pt, {\em Solution of the
  general moment problem via a one-parameter imbedding}, IEEE Transactions on
  Automatic Control, 50 (2005), pp.~811--826.

\bibitem{georgiou2006relative}
\leavevmode\vrule height 2pt depth -1.6pt width 23pt, {\em Relative entropy and
  the multivariable multidimensional moment problem}, IEEE Transactions on
  Information Theory, 52 (2006), pp.~1052--1066.

\bibitem{georgio2003kullback}
{\sc T.T. Georgiou and A.~Lindquist}, {\em {K}ullback-{L}eibler approximation
  of spectral density functions}, IEEE Transactions on Information Theory, 49
  (2003), pp.~2910--2917.

\bibitem{geronimo2004positive}
{\sc J.~S. Geronimo and H.~J. Woerdeman}, {\em Positive extensions,
  {Fej\'er-Riesz} factorization and autoregressive filters in two variables},
  Annals of Mathematics, 160 (2004), pp.~839--906.

\bibitem{grant2008graph}
{\sc M.~Grant and S.~Boyd}, {\em Graph implementations for nonsmooth convex
  programs}, in Recent Advances in Learning and Control, V.~Blondel, S.~Boyd,
  and H.~Kimura, eds., vol.~371 of Lecture Notes in Control and Information
  Sciences, Springer-Verlag, London, 2008, pp.~95--110.

\bibitem{cvx}
\leavevmode\vrule height 2pt depth -1.6pt width 23pt, {\em {CVX}: Matlab
  software for disciplined convex programming, version 2.0 beta}.
\newblock \url{http://cvxr.com/cvx}, Sep. 2013.

\bibitem{kalman1981realization}
{\sc R.E. Kalman}, {\em Realization of covariance sequences}, in Toeplitz
  memorial conference, 1981.
\newblock Tel {A}viv, {I}srael.

\bibitem{karlsson2010theinverse}
{\sc J.~Karlsson, T.T. Georgiou, and A.~Lindquist}, {\em The inverse problem of
  analytic interpolation with degree constraint and weight selection for
  control synthesis}, IEEE Transactions on Automatic Control, 55 (2010),
  pp.~405--418.

\bibitem{karlsson2008stability-preserving}
{\sc J.~Karlsson and A.~Lindquist}, {\em Stability-preserving rational
  approximation subject to interpolation constraints}, IEEE Transactions on
  Automatic Control, 53 (2008), pp.~1724--1730.

\bibitem{karlsson2015themultidimensional}
{\sc J.~Karlsson, A.~Lindquist, and A.~Ringh}, {\em The multidimensional moment
  problem with complexity constraint}, Integral Equations and Operator Theory,
  84 (2016), pp.~395--418.

\bibitem{lang1981spectral}
{\sc S.W. Lang and J.H. McClellan}, {\em Spectral estimation for sensor
  arrays}, in Proceedings of the First ASSP Workshop on Spectral Estimation,
  1981, pp.~3.2.1--3.2.7.

\bibitem{lang1982theextension}
\leavevmode\vrule height 2pt depth -1.6pt width 23pt, {\em The extension of
  {P}isarenko's method to multiple dimensions}, in IEEE International
  Conference on Acoustics, Speech, and Signal Processing (ICASSP), vol.~7, May
  1982, pp.~125--128.

\bibitem{lang1982multidimensional}
\leavevmode\vrule height 2pt depth -1.6pt width 23pt, {\em Multidimensional
  {MEM} spectral estimation}, IEEE Transactions on Acoustics, Speech and Signal
  Processing, 30 (1982), pp.~880--887.

\bibitem{lang1983spectral}
\leavevmode\vrule height 2pt depth -1.6pt width 23pt, {\em Spectral estimation
  for sensor arrays}, IEEE Transactions on Acoustics, Speech and Signal
  Processing, 31 (1983), pp.~349--358.

\bibitem{lev-ari1989multidimensional}
{\sc H.~Lev-Ari, S.~Parker, and T.~Kailath}, {\em Multidimensional
  maximum-entropy covariance extension}, IEEE Transactions on Information
  Theory, 35 (1989), pp.~497--508.

\bibitem{levy1990modeling}
{\sc B.C. Levy, R.~Frezza, and A.J. Krener}, {\em Modeling and estimation of
  discrete-time gaussian reciprocal processes}, IEEE Transactions on Automatic
  Control, 35 (1990), pp.~1013--1023.

\bibitem{lindquist2013onthemultivariate}
{\sc A.~Lindquist, C.~Masiero, and G.~Picci}, {\em On the multivariate
  circulant rational covariance extension problem}, in IEEE 52nd Annual
  Conference on Decision and Control (CDC), 2013, pp.~7155--7161.

\bibitem{lindquist2013thecirculant}
{\sc A.~Lindquist and G.~Picci}, {\em The circulant rational covariance
  extension problem: The complete solution}, IEEE Transactions on Automatic
  Control, 58 (2013), pp.~2848--2861.

\bibitem{LindquistPicci2015}
{\sc A.~Lindquist and G.~Picci}, {\em Linear {S}tochastic {S}ystems: {A}
  {G}eometric {A}pproach to {M}odeling, {E}stimation and {I}dentification},
  vol.~1 of Series in Contemporary Mathematics, Springer-Verlag Berlin
  Heidelberg, 2015.

\bibitem{luenberger1969optimization}
{\sc D.G. Luenberger}, {\em Optimization by Vector Space Methods}, John Wiley
  \& Sons, Inc., New York, 1969.

\bibitem{mahler1962onsome}
{\sc K.~Mahler}, {\em On some inequalities for polynomials in several
  variables}, Journal of the London Mathematical Society, 1 (1962),
  pp.~341--344.

\bibitem{mcclellan1982multi-dimensional}
{\sc J.H. McClellan and S.W. Lang}, {\em Mulit-dimensional {MEM} spectral
  estimation}, in Proceedings of the Institute of Acoustics "Spectral Analysis
  and its Use in Underwater Acoustics": Underwater Acoustics Group Conference,
  Imperial College, London, 29-30 April 1982, 1982, pp.~10.1--10.8.

\bibitem{mcclellan1983duality}
\leavevmode\vrule height 2pt depth -1.6pt width 23pt, {\em Duality for
  multidimensional {MEM} spectral analysis}, Communications, Radar and Signal
  Processing, IEE Proceedings F, 130 (1983), pp.~230--235.

\bibitem{musicus1985maximum}
{\sc B.R. Musicus and A.M. Kabel}, {\em Maximum entropy pole-zero estimation},
  Tech. Report 510, Research Laboratory of Electronics, Massachusetts Institute
  of Technology, August 1985.

\bibitem{nurdin2006new}
{\sc H.I. Nurdin}, {\em New results on the rational covariance extension
  problem with degree constraint}, Systems \& Control Letters, 55 (2006),
  pp.~530 -- 537.

\bibitem{oppenheim1975digital}
{\sc A.V. Oppenheim and R.W. Schafer}, {\em {D}igital {S}ignal {P}rocessing},
  Prentice-Hall, New Jerseys, 1975.

\bibitem{pavon2013geometry}
{\sc M.~Pavon and A.~Ferrante}, {\em On the geometry of maximum entropy
  problems}, SIAM Review, 55 (2013), pp.~415--439.

\bibitem{picci2008modelling}
{\sc G.~Picci and F.P. Carli}, {\em Modelling and simulation of images by
  reciprocal processes}, in Tenth international conference on Computer
  Modelling and Simulation, UKSIM, 2008, pp.~513--518.

\bibitem{ramponi2009aglobally}
{\sc F.~Ramponi, A.~Ferrante, and M.~Pavon}, {\em A globally convergent
  matricial algorithm for multivariate spectral estimation}, IEEE Transactions
  on Automatic Control, 54 (2009), pp.~2376--2388.

\bibitem{rao2014discrete}
{\sc K.R. Rao and P.~Yip}, {\em Discrete cosine transform: algorithms,
  advantages, applications}, Academic press, San Diego, C.A., 1990.

\bibitem{remmert1991theory}
{\sc R.~Remmert}, {\em Theory of complex functions}, Graduate texts in
  mathematics, Springer-Verlag, New York, 1991.
\newblock Translation of: Funktionentheorie I. 2nd ed.

\bibitem{renyi1961measures}
{\sc A.~R\'enyi}, {\em On measures of entropy and information}, in Fourth
  Berkeley symposium on mathematical statistics and probability, vol.~1, 1961,
  pp.~547--561.

\bibitem{ringh2015afast}
{\sc A.~Ringh and J.~Karlsson}, {\em A fast solver for the circulant rational
  covariance extension problem}, in European Control Conference (ECC), July
  2015, pp.~727--733.

\bibitem{ringh2015themultidimensional}
{\sc A.~Ringh, J.~Karlsson, and A.~Lindquist}, {\em The multidimensional
  circulant rational covariance extension problem: Solutions and applications
  in image compression}, in IEEE 54th Annual Conference on Decision and Control
  (CDC), IEEE, 2015, pp.~5320--5327.

\bibitem{ringh2014spectral}
{\sc A.~Ringh and A.~Lindquist}, {\em Spectral estimation of periodic and skew
  periodic random signals and approximation of spectral densities}, in 33rd
  Chinese Control Conference (CCC), 2014, pp.~5322--5327.

\bibitem{rudin1987real}
{\sc W.~Rudin}, {\em Real and {C}omplex {A}nalysis}, McGraw-Hill, New York,
  1987.

\bibitem{schinzel2000polynomials}
{\sc A.~Schinzel}, {\em Polynomials with special regard to reducibility},
  Cambridge University Press, 2000.

\bibitem{shepp1974fourier}
{\sc L.A Shepp and B.F. Logan}, {\em The {F}ourier reconstruction of a head
  section}, IEEE Transactions on Nuclear Science, 21 (1974), pp.~21--43.

\bibitem{stein2003fourier}
{\sc E.M. Stein and R.~Shakarchi}, {\em Fourier analysis: an introduction},
  Princeton University Press, Princeton, N.J., 2003.

\bibitem{stoica1997introduction}
{\sc P.~Stoica and R.~Moses}, {\em Introduction to Spectral Analysis},
  Prentice-Hall, Upper Saddle River, N.J., 1997.

\bibitem{wang2004image}
{\sc Z.~Wang, A.C. Bovik, H.R. Sheikh, and E.P. Simoncelli}, {\em Image quality
  assessment: from error visibility to structural similarity}, IEEE
  Transactions on Image Processing, 13 (2004), pp.~600--612.

\bibitem{woods1976two-dimensionalmark}
{\sc J.W Woods}, {\em Two-dimensional {Markov} spectral estimation}, IEEE
  Transactions on Information Theory, 22 (1976), pp.~552--559.

\bibitem{zorzi2014anewfamily}
{\sc M.~Zorzi}, {\em A new family of high-resolution multivariate spectral
  estimators}, IEEE Transactions on Automatic Control, 59 (2014), pp.~892--904.

\bibitem{zorzi2014rational}
\leavevmode\vrule height 2pt depth -1.6pt width 23pt, {\em Rational
  approximations of spectral densities based on the alpha divergence},
  Mathematics of Control, Signals, and Systems, 26 (2014), pp.~259--278.

\end{thebibliography}

\end{document}